\documentclass[a4paper]
{amsart}%

\newtheorem{theorem}{Theorem}[section]
\newtheorem{lemma}[theorem]{Lemma}
\newtheorem{corollary}[theorem]{Corollary} 

\theoremstyle{definition}

\newtheorem{example}[theorem]{Example}

\theoremstyle{remark}
\newtheorem{remark}[theorem]{Remark}
\newtheorem{ack}{Acknowledgment} 

\numberwithin{equation}{section}

\usepackage[dvipdfmx]{graphicx}
\usepackage{amssymb}

\def\diaCrossP{\unitlength.1em
  \begin{minipage}{15\unitlength}
    \begin{picture}(15,15)
      \put(0,0){\vector(1,1){15}}
      \qbezier(15,0)(15,0)(10,5)
      \qbezier(5,10)(0,15)(0,15)
      \put(0,15){\vector(-1,1){0}}
    \end{picture}
  \end{minipage}
}
\def\diaCrossN{\unitlength.1em
  \begin{minipage}{15\unitlength}
    \begin{picture}(15,15)
      \put(15,0){\vector(-1,1){15}}
      \qbezier(0,0)(0,0)(5,5)
      \qbezier(10,10)(15,15)(15,15)
      \put(15,15){\vector(1,1){0}}
    \end{picture}
  \end{minipage}
}
\def\diaSmooth{\unitlength.1em
  \begin{minipage}{15\unitlength}
    \begin{picture}(15,15)
      \qbezier(0,0)(10,7.5)(0,15)
      \qbezier(15,0)(5,7.5)(15,15)
      \put(15,15){\vector(1,1){0}}
      \put(0,15){\vector(-1,1){0}}
    \end{picture}
  \end{minipage}
}
\def\diaTrivial{\unitlength.1em
  \begin{minipage}{10\unitlength}
    \begin{picture}(10,10)
      \put(5,5){\circle{10}}
    \end{picture}
  \end{minipage}
}

\def\diaProjected{\unitlength.1em
  \begin{minipage}{14\unitlength}
    \begin{picture}(14,14)
      \put(0,0){\vector(1,1){14}}
      \put(14,0){\vector(-1,1){14}}
    \end{picture}
  \end{minipage}
}

\def\diaPosTouchR{\unitlength.1em
  \begin{minipage}{15\unitlength}
    \begin{picture}(15,15)
      \qbezier(0,0)(15,15)(15,15)
      \qbezier(15,0)(15,0)(10,5)
      \qbezier(5,10)(0,15)(0,15)
    \put(10,5){$\ast$}
    \end{picture}
  \end{minipage}
}
\def\diaPosTouchL{\unitlength.1em
  \begin{minipage}{15\unitlength}
    \begin{picture}(15,15)
      \qbezier(0,0)(15,15)(15,15)
      \qbezier(15,0)(15,0)(10,5)
      \qbezier(5,10)(0,15)(0,15)
    \put(0,5){$\ast$}
    \end{picture}
  \end{minipage}
}
\def\diaPosTouchLR{\unitlength.1em
  \begin{minipage}{15\unitlength}
    \begin{picture}(15,15)
      \qbezier(0,0)(15,15)(15,15)
      \qbezier(15,0)(15,0)(10,5)
      \qbezier(5,10)(0,15)(0,15)
    \put(10,5){$\ast$}
    \put(0,5){$\ast$}
    \end{picture}
  \end{minipage}
}
\def\diaNegTouchB{\unitlength.1em
  \begin{minipage}{15\unitlength}
    \begin{picture}(15,15)
      \qbezier(0,0)(15,15)(15,15)
      \qbezier(15,0)(15,0)(10,5)
      \qbezier(5,10)(0,15)(0,15)
    \put(5,0){$\ast$}
    \end{picture}
  \end{minipage}
}
\def\diaNegTouchT{\unitlength.1em
  \begin{minipage}{15\unitlength}
    \begin{picture}(15,15)
      \qbezier(0,0)(15,15)(15,15)
      \qbezier(15,0)(15,0)(10,5)
      \qbezier(5,10)(0,15)(0,15)
    \put(5,10){$\ast$}
    \end{picture}
  \end{minipage}
}
\def\diaNegTouchTB{\unitlength.1em
  \begin{minipage}{15\unitlength}
    \begin{picture}(15,15)
      \qbezier(0,0)(15,15)(15,15)
      \qbezier(15,0)(15,0)(10,5)
      \qbezier(5,10)(0,15)(0,15)
    \put(5,10){$\ast$}
    \put(5,0){$\ast$}
    \end{picture}
  \end{minipage}
}

\makeatletter\@namedef{subjclassname@2020}{%
 \textup{2020} Mathematics Subject Classification}
\makeatother

\begin{document}
\title[Integral region choice problems]
{Integral region choice problems \\ on link diagrams}
\author{Tomomi Kawamura}
\address{
Graduate school of Mathematics, 
Nagoya University \\
Furocho, Chikusaku, Nagoya 464-8602, JAPAN }
\email{
tomomi@math.nagoya-u.ac.jp 
}

\subjclass[2020]{57K10}
\keywords{
region choice problem, Alexander numbering, checkerboard coloring, 
region choice matrix
}
\date{\today}

\begin{abstract}
Shimizu introduced a region crossing change unknotting operation for knot diagrams. 
As extensions, 
two integral region choice problems were proposed and 
the existences of solutions of the problems were shown for all non-trivial knot diagrams 
by Ahara and Suzuki, and Harada. 
We relate both integral region choice problems with 
an Alexander numbering for regions of a link diagram,
and  give alternative proofs of 
the existences of solutions for knot diagrams. 
We also discuss   
the problems on link diagrams. 
For each of the problems on the diagram of a two-component link, 
we give a neccessary and sufficient condition that 
there exists a solution. 
\end{abstract}

\maketitle


\section{Introduction}\label{sect;intro}

A \emph{link} is a closed 1-manifold 
smoothly embedded in the 3-space $\mathbb{R}^3$ or in the 3-sphere $S^3$ 
and a \emph{knot} is a link with one component. 
A link in the 3-space is presented as the natural projection image on the 2-plane $\mathbb{R}^2$ 
where the singular points are transverse double points with over/under information. 
This presentation  is called a \emph{link diagram} or a \emph{diagram} of the link. 
A \emph{diagram} of a link in the 3-sphere $S^3=\mathbb{R}^3\cup \{ \infty\}$ 
is given on the 2-sphere $S^2=\mathbb{R}^2\cup \{ \infty \}$ similarly. 
For each link diagram, a connected component of the complement of the projection image on $\mathbb{R}^2$ or $S^2$ 
is called a \emph{region}. 

In \cite{shimizu}, 
Shimizu 
defined a \emph{region crossing change} at a region for a diagram
to be the crossing change at all the crossings on the boundary of the region 
as an unknotting operation for a knot diagram, 
which was proposed by Kengo Kishimoto. 
For example in Figure \ref{Fig;RCC4_1}, 
the left diagram is changed to the right  diagram, 
choosing the region marked with $\ast$ as illustrated on the middle 
and changing the three crossings on the boundary of the marked region.  
In \cite{cheng, chenggao}, 
Cheng and Gao 
gave a necessary and sufficient condition 
that a region crossing change is an unknotting operation on a link diagram. 

\begin{figure}[htbp]
\begin{center}
\includegraphics[height=3cm,clip]{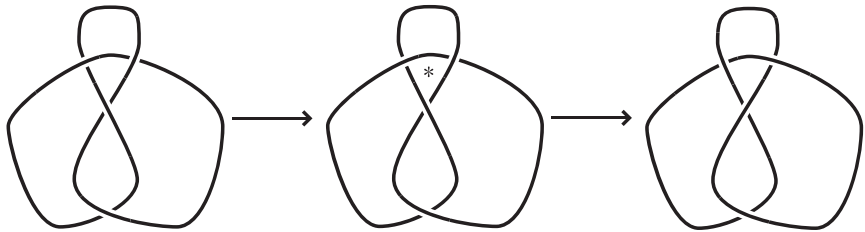}
\end{center}
\caption{An example of a region crossing change.}
\label{Fig;RCC4_1}
\end{figure}

It is known that a region crossing change can be interpreted as follows. 
We call a diagram ignored over/under information a \emph{projection}. 
Let each crossing of the given projection be equipped with a score $0$ or $1$ modulo $2$. 
We choose a region of the projection. 
Then the scores of all the crossings on its boundary are increased by $1$ modulo $2$. 
For example, 
the region crossing change illustrated on Figure \ref{Fig;RCC4_1} is interpreted 
as Figure \ref{Fig;Z2RC4_1}. 
Shimizu showed that the scores of all the crossings on any knot diagram become $0$ 
by some choices of regions.
Cheng and Gao induced a $\mathbb{Z}_2$-homomorphism 
from region crossing changes on link diagrams. 
In \cite{hashizume2013, hashizume2015}, 
Hashizume studied structures of their $\mathbb{Z}_2$-homomorphism.  

\begin{figure}[htbp]
\begin{center}
\includegraphics[height=3cm,clip]{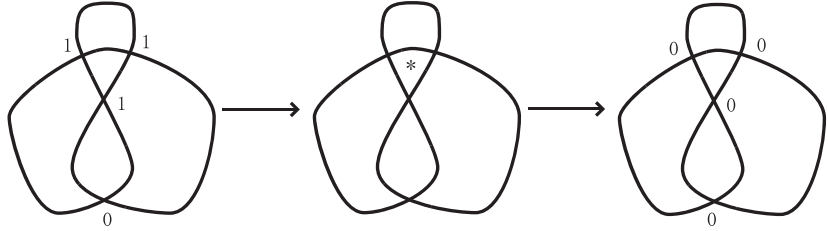}
\end{center}
\caption{An example of an interpreted region crossing change.}
\label{Fig;Z2RC4_1}
\end{figure}

As an extension of a region crossing  change to an integral range, 
Ahara and Suzuki proposed an integral region choice problem  
and showed the existence of  a solution of this problem for all knot projections 
in \cite{aharasuzuki}. 
Let each crossing of the given projection be equipped with an integral score. 
We choose a region of the projection and assign an integer $u$ to it. 
Then the scores of all the crossings on its boundary are increased by $u$. 
For example in Figure \ref{Fig;DZRC4_1}, 
the scores of the crossings on the left projection are changed to the right, 
assigning integers to regions as the middle projection; 
$1\mapsto 1+0+2+(-1)+(-2)=0$,  $-1\mapsto -1+0+0+(-1)+2=0$, 
$3\mapsto 3+0+0+(-2)+(-1)=0$, and $2\mapsto 2+0+0+0+(-2)=0$. 
Ahara and Suzuki showed that 
the scores of all the crossings on any knot projection become $0$ 
by some choices of regions and some assignments of the integers to them. 
We shall call their problem  
a \emph{definite integral region choice problem}. 
In Section \ref{sect;DZRCP}, we state their result exactly. 

\begin{figure}[htbp]
\begin{center}
\includegraphics[height=3cm,clip]{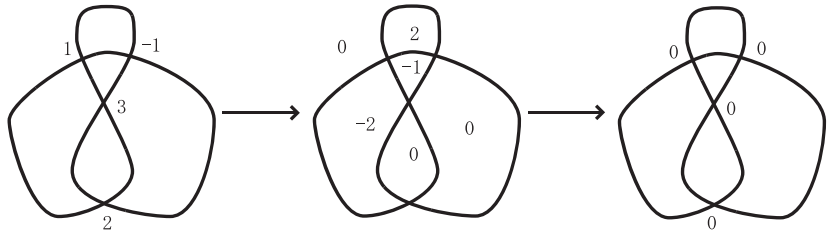}
\end{center}
\caption{An example of a definite integral region choice problem.}
\label{Fig;DZRC4_1}
\end{figure}

By an argument similar to that due to Ahara and Suzuki, 
Harada showed in his master thesis \cite{haradaM} that 
there exists a solution of an alternating integral region choice problem for all knot diagrams, 
which was suggested by Yasuyoshi Yonezawa. 
Let each crossing of the given diagram be equipped with an integral score. 
We choose a region of the projection and assign an integer $u$ to it. 
Then the score of each crossing on its boundary is changed as follows.  
If the region touches the crossing at the corner contained by the underpass and the overpass counterclockwise, 
that means 
$\diaPosTouchL$ or $\diaPosTouchR$ , 
the score of  the crossing is increased by $u$. 
If the region touches the crossing at the corner contained by the underpass and the overpass clockwise, 
that means 
$\diaNegTouchT$ or $\diaNegTouchB$ , 
the score of  the crossing is decreased by $u$. 
For example in Figure \ref{Fig;AZRC4_1}, 
the scores of the crossings on the left diagram are changed to the right, 
assigning integers to regions as the middle diagram; 
$1\mapsto 1+0-(-2)+(-1)-2=0$,  $-1\mapsto -1+0-0+(-1)-(-2)=0$, 
$3\mapsto 3-0+0-2+(-1)=0$, and $2\mapsto 2+0-0+0-2=0$. 
Harada showed that 
the scores of all the crossings on any knot diagram become $0$ 
by some choices of regions and some assignments of the integers to them. 
We shall call this proposed problem as another extension of a region crossing change
an \emph{alternating integral region choice problem}. 
In Section \ref{sect;AZRCP}, we state his result exactly. 

\begin{figure}[htbp]
\begin{center}
\includegraphics[height=3cm,clip]{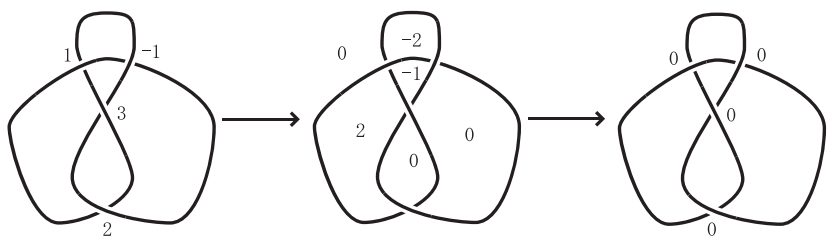}
\end{center}
\caption{An example of an alternating integral region choice problem.}
\label{Fig;AZRC4_1}
\end{figure}

In \cite{aharasuzuki, haradaM}, 
Ahara, Suzuki and Harada reduced the above integral region choice problems to systems of  linear equations, 
as explained in Section \ref{sect;DZRCP} and \ref{sect;AZRCP} in this article, 
and they showed the existences of  solutions for non-trivial knot diagrams. 
We show that 
an Alexander numbering for regions of a link diagram is 
a solution of the system of homogeneous linear equations 
reduced from an alternating integral region choice problem in Section \ref{sect;kernel}. 
By this result, we give alternative proofs of the existences of solutions of 
both alternating and definite integral region choice problems 
for all non-trivial knot diagrams in Section \ref{sect;AZRCPpf} and \ref{sect;DZRCPpf}. 

In \cite{shimizu}, 
Shimizu used checkerboard colorings to regions of knot diagrams 
for showing that a region crossing change is an unknotting operation . 
Cheng and Gao \cite{chenggao}, and Hashizume \cite{hashizume2013, hashizume2015} 
also used checkerboard colorings
for discussing region crossing changes on link diagrams. 
An Alexander numbering is an integral extension of a checkerboard coloring, 
as mentioned in Section \ref{sect;pre}. 
In this article, 
we use Alexander numberings to discuss the integral region choice problems on link diagrams, 
which are integral extensions of region crossing changes.  

In Section \ref{sect;RCmatrix} and \ref{sect;RCmatrixS},  
we determine the ranks for the coefficient matrices
of the systems of linear equations reduced from the integral region choice problems,  
applying the arguments in the original proofs of 
the solvability of integral region choice problems on knot diagrams
in \cite{aharasuzuki, haradaM} to link diagrams.   
Then we obtain an extension of  the result about the incidence matrix due to Cheng and Gao \cite{chenggao}. 

In Section \ref{sect;basis}, 
we give a basis of the space of 
solutions of the system of homogeneous linear equations 
reduced from each of integral region choice problems on link diagrams.  
In Section \ref{sect;image}, 
we give neccessary and sufficient conditions that 
there exist solutions of integral region choice problems 
on the connected diagram of a two-component link.  
These results are extentions of some of the results about region crossing changes on link diagrams 
due to Cheng and Gao \cite{chenggao}, and Hashizume \cite{hashizume2013, hashizume2015}.

\begin{ack}
The author would like to thank Megumi Hashizume 
for giving valuable advices and a lot of information about a region crossing change. 
The author also would like to thank Yasuyoshi Yonezawa 
for his suggestion and 
Shingo Harada for his works on an alternating integral region choice problem. 
She is also grateful to Akio Kawauchi 
for his helpful comments. 
\end{ack}

\section{Preliminary}\label{sect;pre}


By the Jordan curve theorem, 
any short arc without a crossing on a link diagram 
lies on the intersection of just two boundaries of regions. 
Each crossing is touched by at most four regions. 
If the number of the regions touching the fixed crossing is less than four, 
it must be three and 
the pair of the corners of the same region touching the crossing 
are not adjacent each other around the crossing. 
This fact is also shown from the Jordan curve theorem. 
In this case, such a crossing is called a \emph{reducible} crossing. 
If a link diagram have a reducible crossing,  it is called a \emph{reducible} diagram. 
Otherwise, it is called an \emph{irreducible} diagram. 

\begin{lemma}[cf. \cite{aharasuzuki, hashizume2015}]\label{lem;Euler}
Let $D$ be a link diagram or projection. 
If $D$ has $d$ connected components and $n$ crossings, 
then it has $n+d+1$ regions. 
\end{lemma}

\begin{proof}
It is shown by the Euler formula. 
\end{proof}

On an oriented link diagram $D$, 
we say that we \emph{splice} at a crossing $x$ 
if we change the diagram $D$ around the crossing $x$, $\diaCrossP$ or $\diaCrossN$, to $\diaSmooth$ 
and obtain the new link diagram $D_x$.   
This local move between oriented link diagrams 
is called a \emph{splicing} or \emph{smoothing} at the crossing $x$. 
The change from $D_x$ to $D$ is called an \emph{unsplicing} at $x$. 
In this article, 
the local moves $\diaProjected$ to $\diaSmooth$ and $\diaSmooth$ to $\diaProjected$
among oriented link projections are also called 
a \emph{splicing} and an \emph{unsplicing} respectively. 

In \cite{alexander}, 
Alexander assigned an integer index to each region of an oriented link diagram or projection, 
so that for any oriented arc on the link diagram,  
an index of the left region adjacent to the arc is larger that of the right by one. 
Such an index is called an \emph{Alexander index}, 
and this assignment of the indexes is called an \emph{Alexander indexing} or an \emph{Alexander numbering}. 
In \cite{kauffman2006}, 
Kauffman also defined an Alexander indexing for an oriented link projection 
and show that there exist an Alexander indexing for any projection, 
though an index of the right region is assigned larger than that of the left by one  
for any oriented arc on the link projection. 

It is known that 
we can shade regions  for any link projection 
so that each two regions adjacent by an arc on the projection are shaded and unshaded, 
and such shading is call a \emph{checkerboard coloring}. 
For any oriented link diagram or projection, 
if we shade only the regions assigned odd number by an Alexander numbering, 
then we obtain a checkerboard coloring.   
If we reverse the orientation of some link components fixing a region and its index, 
we obtain a new Alexander numbering and 
the same checkerboard coloring. 
In this article, 
we shall call an Alexander numbering modulo $2$ a \emph{checkerboard coloring}.


\section{A definite integral region choice problem}\label{sect;DZRCP}


Let $D$ be a link diagram or projection 
with $d$ connected components and $n$ crossings $x_1, \cdots , x_n$, $n\geq 1$. 
We note that $d$ is not greater than the number of the link components. 
Let $R_1, \cdots , R_{n+d+1}$ be the regions of $D$. 
In \cite{aharasuzuki}, 
Ahara and Suzuki induced two region choice matrices $A_{d1}(D)$ and $A_{d2}(D)$ 
with $n$ rows and $n+d+1$ columns as follows, 
where they denoted them by $A_1(D)$ and $A_2(D)$. 
We determine each element $a^{(d1)}_{ij}$ by 
\[
a^{(d1)}_{ij}=
\begin{cases}
1 & \text{if $x_i \in \partial R_j$}, \\
0 & \text{if $x_i \notin \partial R_j$} .
\end{cases}
\]
The \emph{region choice matrix of the single counting rule} for $D$ is 
the matrix $A_{d1}(D)$ with the element $a^{(d1)}_{ij}$ on the $i$-th row and the $j$-th column. 
We determine each element $a^{(d2)}_{ij}$ by 
\[
a^{(d2)}_{ij}=
\begin{cases}
2 & \text{if $R_j$ touches $x_i$ twice}, \\
a^{(d1)}_{ij} & \text{otherwise} .
\end{cases}
\]
The \emph{region choice matrix of the double counting rule} for $D$ is 
the matrix $A_{d2}(D)$ with the element $a^{(d2)}_{ij}$ on the $i$-th row and the $j$-th column. 
We shall call these two region choice matrices by 
the \emph{definite region choice matrices}. 

Using the definite region choice matrices, 
the definite integral region choice problem and the existence of solutions for it 
are stated as follows.  

\begin{theorem}[\cite{aharasuzuki}]\label{thm;aharasuzuki;mat}
Let $D$ be a knot diagram or projection with $n$ crossings $x_1, \cdots , x_n$, $n\geq 1$. 
Let $R_1, \cdots , R_{n+2}$ be the regions of $D$. 
\begin{enumerate}
\item
Let $A_{d1}(D)$ be the definite region choice matrix of the single counting rule for $D$. 
For any $\mathbf{c}\in \mathbb{Z}^n$, 
there exists a solution $\mathbf{u}\in \mathbb{Z}^{n+2}$ 
such that $A_{d1}(D)\mathbf{u}+\mathbf{c}=\mathbf{0}$. 
\item
Let $A_{d2}(D)$ be the definite region choice matrix of the double counting rule for $D$. 
For any $\mathbf{c}\in \mathbb{Z}^n$, 
there exists a solution $\mathbf{u}\in \mathbb{Z}^{n+2}$ 
such that $A_{d2}(D)\mathbf{u}+\mathbf{c}=\mathbf{0}$. 
\end{enumerate}
\end{theorem}

\begin{example}
Let $D$ be the knot projection given in  Figure \ref{Fig;DZRC4_1}. 
Under certain orders of crossings and regions, we have 
\[
A_{d1}(D)=A_{d2}(D)=
\begin{pmatrix}
1 & 1 & 1 & 0 & 0 & 1 \\
1 & 1 & 0 & 0 & 1 & 1 \\
0 & 1 & 1 & 1 & 1 & 0 \\
0 & 0 & 1 & 1 & 1 & 1
\end{pmatrix}
.
\]
Figure \ref{Fig;DZRC4_1} implies the equation   
\[
A_{di}(D)
\begin{pmatrix}
2 \\ -1 \\ -2 \\ 0 \\ 0 \\ 0 
\end{pmatrix}
+
\begin{pmatrix}
1 \\ -1 \\ 3 \\ 2 
\end{pmatrix}
=
\begin{pmatrix}
0 \\ 0 \\ 0 \\ 0
\end{pmatrix}
\]
holds for $i=1, 2$.
\end{example}

If we transpose the incidence matrices 
induced by Cheng and Gao \cite{chenggao} and Hashizume \cite{hashizume2013}, 
it is same as the definite region choice matrix of the single counting rule 
modulo $2$ up to permutations of rows and columns. 

\section{An alternating integral region choice problem}\label{sect;AZRCP}


Let $D$ be a link diagram with $d$ connected components and $n$ crossings $x_1, \cdots , x_n$, $n\geq 1$. 
Let $R_1, \cdots , R_{n+d+1}$ be the regions of $D$. 
In \cite{haradaM}, 
Harada induced two region choice matrices $A_{a1}(D)$ and $A_{a2}(D)$ 
with $n$ rows and $n+d+1$ columns as follows, 
where he denoted them by $B_1(D)$ and $B_2(D)$. 

We determine each elements $a^{(a1)}_{ij}$ as follows. 
We define $a^{(a1)}_{ij}=1$, 
if the region $R_j$ touches the crossing $x_i$ at the corner contained by the underpass and the overpass counterclockwise, 
that means 
$\diaPosTouchL$ or $\diaPosTouchR$ . 
We define $a^{(a1)}_{ij}=-1$, 
if $R_j$ touches $x_i$ at the corner contained by the underpass and the overpass clockwise, 
that means 
$\diaNegTouchT$ or $\diaNegTouchB$ . 
If $x_i$ does not lie on $\partial R_j$, we define $a^{(a1)}_{ij}=0$. 
The \emph{alternating region choice matrix of the single counting rule} for $D$ is 
the matrix $A_{a1}(D)$ with the element $a^{(a1)}_{ij}$ on the $i$-th row and the $j$-th column. 
We determine each element $a^{(a2)}_{ij}$ by 
\[
a^{(a2)}_{ij}=
\begin{cases}
2a^{(a1)}_{ij} & \text{if $R_j$ touches $x_i$ twice as $\diaPosTouchLR$ or $\diaNegTouchTB$} , \\
a^{(a1)}_{ij} & \text{otherwise} .
\end{cases}
\]
The \emph{alternating region choice matrix of the double counting rule} for $D$ is 
the matrix $A_{a2}(D)$ with the element $a^{(a2)}_{ij}$ on the $i$-th row and the $j$-th column. 

We compare the definitions of 
$a_{ij}^{(d1)}, a_{ij}^{(d2)}, a_{ij}^{(a1)}, a_{ij}^{(a2)}$ on Table 1, 
where the region $R_j$ includes the corners marked with $\ast$ 
but does not include the unmarked corners 
around the crossing $x_i$. 

\begin{table}[htbp]
\begin{center}
\begin{tabular}{l| c c c c c}
$x_i$ and $Rj$ & \ $\diaPosTouchLR$ \  & \ $\diaPosTouchL$ or $\diaPosTouchR$ \  
& \ $\diaNegTouchT$ or $\diaNegTouchB$ \  & \ $\diaNegTouchTB$ \  & otherwise \\[7pt]
\hline 
$a_{ij}^{(d1)}$ & $1$ & $1$ & $1$ & $1$ & $0$ \\[5pt]
$a_{ij}^{(d2)}$ & $2$ & $1$ & $1$ & $2$ & $0$ \\[5pt]
$a_{ij}^{(a1)}$ & $1$ & $1$ & $-1$ & $-1$ & $0$ \\[5pt]
$a_{ij}^{(a2)}$ & $2$ & $1$ & $-1$ & $-2$ & $0$ 
\end{tabular}
\end{center}
\caption{
$a_{ij}^{(d1)}, a_{ij}^{(d2)}, a_{ij}^{(a1)}, a_{ij}^{(a2)}$.}
\end{table}

Using the alternating region choice matrices, 
the alternating integral region choice problem and the existence of solutions for it 
are stated as follows.  

\begin{theorem}[\cite{haradaM}]\label{thm;harada;mat}
Let $D$ be a knot diagram with $n$ crossings $x_1, \cdots , x_n$, $n\geq 1$. 
Let $R_1, \cdots , R_{n+2}$ be the regions of $D$. 
\begin{enumerate}
\item
Let $A_{a1}(D)$ be the alternating region choice matrix of the single counting rule for $D$. 
For any $\mathbf{c}\in \mathbb{Z}^n$, 
there exists a solution $\mathbf{u}\in \mathbb{Z}^{n+2}$ 
such that $A_{a1}(D)\mathbf{u}+\mathbf{c}=\mathbf{0}$. 
\item
Let $A_{a2}(D)$ be the alternating region choice matrix of the double counting rule for $D$. 
For any $\mathbf{c}\in \mathbb{Z}^n$, 
there exists a solution $\mathbf{u}\in \mathbb{Z}^{n+2}$ 
such that $A_{a2}(D)\mathbf{u}+\mathbf{c}=\mathbf{0}$. 
\end{enumerate}
\end{theorem}

\begin{example}
Let $D$ be the knot diagram given in Figure \ref{Fig;AZRC4_1}. 
Under certain orders of crossings and regions, we have 
\[
A_{a1}(D)=A_{a2}(D)=
\begin{pmatrix}
-1 & 1 & -1 & 0 & 0 & 1 \\
-1 & 1 & 0 & 0 & -1 & 1 \\
0 & 1 & -1 & 1 & -1 & 0 \\
0 & 0 & -1 & 1 & -1 & 1
\end{pmatrix}
.
\]
Figure \ref{Fig;AZRC4_1} implies the equation   
\[
A_{ai}(D)
\begin{pmatrix}
-2 \\ -1 \\ 2 \\ 0 \\ 0 \\ 0 
\end{pmatrix}
+
\begin{pmatrix}
1 \\ -1 \\ 3 \\ 2 
\end{pmatrix}
=
\begin{pmatrix}
0 \\ 0 \\ 0 \\ 0
\end{pmatrix}
\]
holds for $i=1, 2$.
\end{example}

\begin{remark}
If we transpose the incidence matrix 
induced by Cheng and Gao \cite{chenggao} and Hashizume \cite{hashizume2013}, 
it is same as the alternating region choice matrix of the single counting rule 
modulo $2$ up to permutations of rows and columns. 
\end{remark}

\begin{remark}
In this article, we reverse signs of the elements in 
the alternating region choice matrices defined by Harada \cite{haradaM}, 
since our alternating region choice matrix of the double counting rule 
coincides with the Alexander matrix 
defined in \cite{alexander} 
if we substitute $1$ for the variable. 
In \cite{kauffman2006}, 
Kauffman illustrated the definition of the Alexander matrix as a crossing with labeled corners   
  \begin{minipage}{15\unitlength}
    \begin{picture}(15,15)
      \qbezier(0,0)(15,15)(15,15)
      \qbezier(15,0)(15,0)(10,5)
      \qbezier(5,10)(0,15)(0,15)
      \put(0,15){\vector(-1,1){0}}
    \put(12,5){\footnotesize $t$}\put(2,12){\footnotesize $-t$}
    \put(0,5){\footnotesize $1$}\put(2,-2){\footnotesize $-1$}
    \end{picture}
  \end{minipage}
. 
In his terms, 
our alternating region choice matrix  and the definite region choice matrix of the double counting rule are denoted by 
  \begin{minipage}{15\unitlength}
    \begin{picture}(15,15)
      \qbezier(0,0)(15,15)(15,15)
      \qbezier(15,0)(15,0)(10,5)
      \qbezier(5,10)(0,15)(0,15)
    \put(12,5){\footnotesize $1$}\put(2,12){\footnotesize $-1$}
    \put(0,5){\footnotesize $1$}\put(2,-2){\footnotesize $-1$}
    \end{picture}
  \end{minipage}
 and 
  \begin{minipage}{15\unitlength}
    \begin{picture}(15,15)
      \qbezier(0,0)(15,15)(15,15)
      \qbezier(15,0)(15,0)(10,5)
      \qbezier(5,10)(0,15)(0,15)
    \put(10,5){\footnotesize $1$}\put(5,12){\footnotesize $1$}
    \put(0,5){\footnotesize $1$}\put(5,-2){\footnotesize $1$}
    \end{picture}
  \end{minipage}
 respectively. 
In \cite{kawauchiEdu2014}, 
Kawauchi indicated that 
the transposed incidence matrix is same as 
the Alexander matrix substituted $1$ modulo $2$, 
and that the solvability of the original region choice problem on knot diagrams is induced 
by the fact the Alexander polynomial substituted $1$ becomes $1$ for any knot. 
This fact also implies that Theorem \ref{thm;harada;mat} (2).  
\end{remark}

We give more examples to compare 
definite and alternating region choice matrices 
of the single counting rule 
and of the double counting rule. 

\begin{example}
Let $D$ be the link diagram given as 
the split sum of the $l$ copies of the knot diagram with only one crossing 
such that one region touches all crossings twice. 
The diagram $D$ represents a trivial $l$-component link.  
Under certain orders of crossings and regions, 
we obtain 
\[
A_{d1}(D)=
\begin{pmatrix}
1        & 1 & 1 & 0 & 0 & 0 & \ldots & 0  \\
1        & 0 & 0 & 1 & 1 & 0 & \ldots & 0\\
\vdots &   &   &   &   & \ddots &   &   \\
1        & 0 & 0 & 0 & 0 & \ldots  & 1 & 1 
\end{pmatrix}
,
\]
\[
A_{d2}(D)=
\begin{pmatrix}
2        & 1 & 1 & 0 & 0 & 0 & \ldots & 0  \\
2        & 0 & 0 & 1 & 1 & 0 & \ldots & 0\\
\vdots &   &   &   &   & \ddots &   &   \\
2        & 0 & 0 & 0 & 0 & \ldots  & 1 & 1 
\end{pmatrix}
,
\]
and 
\[
A_{a1}(D)=
\begin{pmatrix}
-\varepsilon _1  & \varepsilon _1 & \varepsilon _1 & 0 & 0 & 0 & \ldots & 0  \\
-\varepsilon _2  & 0 & 0 & \varepsilon _2 & \varepsilon _2 & 0 & \ldots & 0\\
\vdots             &   &   &   &   & \ddots &   &   \\
-\varepsilon _l   & 0 & 0 & 0 & 0 & \ldots  & \varepsilon _l  & \varepsilon _l  
\end{pmatrix}
,
\]
\[
A_{a2}(D)=
\begin{pmatrix}
-2\varepsilon _1 & \varepsilon _1 & \varepsilon _1 & 0 & 0 & 0 & \ldots & 0  \\
-2\varepsilon _2 & 0 & 0 & \varepsilon _2 & \varepsilon _2 & 0 & \ldots & 0\\
\vdots             &   &   &   &   & \ddots &   &   \\
-2\varepsilon _l  & 0 & 0 & 0 & 0 & \ldots  & \varepsilon _l  & \varepsilon _l  
\end{pmatrix}
,
\]
where $\varepsilon _i=1$ if the $i$-th crossing is positive $\diaCrossP$, 
$\varepsilon _i=-1$ if it is negative $\diaCrossN$. 
Each of these matrices has $l$ rows and $2l+1$ columns. 
\end{example}

\begin{example}\label{ex;n1}
Let $D$ be the link diagram given as 
the split sum of the knot diagram with only one crossing and the $l-1$ copies of the trivial knot daigaram $\diaTrivial$. 
The diagram $D$ represents a trivial $l$-component link.  
On $D$ with  certain orders of regions, 
we obtain 
\[
A_{d1}(D)=
\begin{pmatrix}
1        & 1 & 1 & 0 & 0 & 0 & \ldots & 0  
\end{pmatrix}
,
\]
\[
A_{d2}(D)=
\begin{pmatrix}
2        & 1 & 1 & 0 & 0 & 0 & \ldots & 0 
\end{pmatrix}
,
\]
and 
\[
A_{a1}(D)=
\begin{pmatrix}
-\varepsilon & \varepsilon & \varepsilon & 0 & 0 & 0 & \ldots & 0  
\end{pmatrix}
,
\]
\[
A_{a2}(D)=
\begin{pmatrix}
-2\varepsilon & \varepsilon & \varepsilon & 0 & 0 & 0 & \ldots & 0  
\end{pmatrix}
,
\]
where $\varepsilon =1$ if the crossing is positive, otherwise $\varepsilon =-1$, 
and the number of $0$ appearing on each matrix is $l-1$. 
\end{example}

\section{Kernel solutions from Alexander numberings}\label{sect;kernel}


Let $D$ be a link diagram with $d$ connected components 
and $n$ crossings $x_1, \cdots , x_n$, $n\geq 1$. 
Let $R_1, \cdots , R_{n+d+1}$ be the regions of $D$. 
Let all crossings be equipped with $0$. 
Then the integral region choice problems induce $\mathbb{Z}$-homomorphisms. 
We denote 
by $\Phi _{di}(D) : \mathbb{Z}^{n+d+1}\rightarrow \mathbb{Z}^{n}$ 
and $\Phi _{ai}(D) : \mathbb{Z}^{n+d+1}\rightarrow \mathbb{Z}^{n}$
the induced homomorphisms with representation matrices $A_{di}(D)$ and $A_{ai}(D)$ respectively, 
$i=1,2$. 
We call a vector $\mathbf{u} \in \mathbb{Z}^{n+d+1}$ 
with $A_{d1} (D)\mathbf{u}=\mathbf{0}$ (resp. $A_{d2}(D)\mathbf{u}=\mathbf{0}$) 
a \emph{kernel solution} for the definite region choice matrix 
of the single (resp. double) counting rule, 
similarly to that defined to knot projections in \cite{aharasuzuki}.  
We call a vector $\mathbf{u} \in \mathbb{Z}^{n+d+1}$ 
with $A_{a1} (D)\mathbf{u}=\mathbf{0}$ (resp. $A_{a2}(D)\mathbf{u}=\mathbf{0}$) 
a \emph{kernel solution} for the alternating region choice matrix 
of the single (resp. double) counting rule, 
similarly to that defined to knot diagrams in \cite{haradaM}.  

\begin{lemma}\label{lem;AlexindKer}
On any link diagram with at least one crossing, 
an Alexander numbering for an arbitrary orientation 
gives a kernel solution for 
an alternating region choice matrix of the double counting rule. 
\end{lemma}

\begin{proof}
On  the given oriented link diagram $D$, 
we fix an Alexander numbering for it. 
We take an arbitrary crossing $x$ of $D$.  
We may assume that 
$x$ lies as $\diaCrossP$ or $\diaCrossN$ in $D$. 
We suppose that the index of the right region of $x$ is $p\in \mathbb{Z}$. 
Then the index of the left region of $x$ is $p+2$ 
and the rest regions touching $x$ are $p+1$.  
We have $p-(p+1)+(p+2)-(p-1)=0$ and $-p+(p+1)-(p+2)+(p-1)=0$. 
Then the  alternating region choice obtained from the Alexander numbering 
does not change the scores of the crossings. 
\end{proof}

Let $D$ be an oriented link diagram with ordered link components, 
and $D_i$ be a sub-diagram of $D$ representing $i$-th link component, $i=1, \cdots , l$. 
We fix a sub-diagram $D_i$. 
We ignore the diagrams of  link components other $D_i$, 
and take an Alexander numbering. 
Each region $R$ of the diagram $D$ is a subset of one region $S$ of the diagram $D_i$.  
Let $a_S$ be the integer assigned to $S$ by this Alexander numbering. 
We assign the integer $a_S$ to the region $R$ and 
denote it by $u_R$. 
We call this assignment of the integers to the region $\{ u_R\} _R$ 
a \emph{componentwise Alexander numbering associated with} $D_i$.  
Figure \ref{Fig;compAlex} gives 
an example of a pair of componentwise Alexander numberings  
on a 2-component link diagram. 

\begin{figure}[htbp]
\begin{center}
\includegraphics[height=3.5cm,clip]{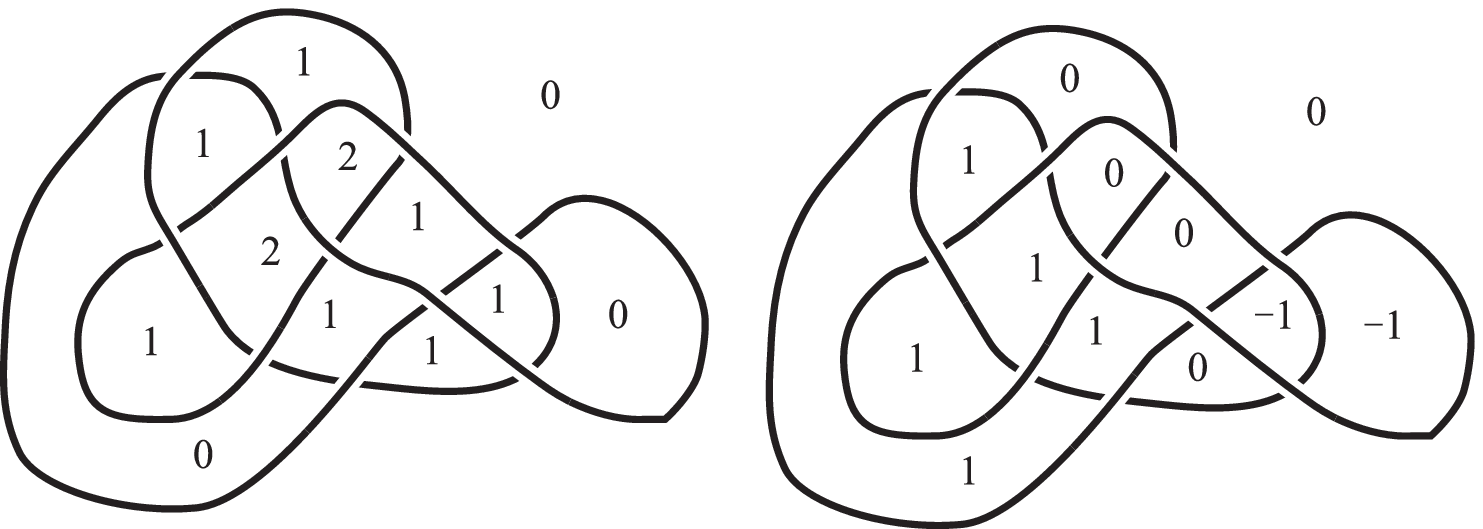}
\end{center}
\caption{Componentwise Alexander numberings.}
\label{Fig;compAlex}
\end{figure}

\begin{lemma}\label{lem;compAlex}
On any oriented link diagram with at least one crossing, 
each componentwise Alexander numbering gives a kernel solution for 
an alternating region choice matrix of the double counting rule. 
\end{lemma}

\begin{proof}
Let $D$ be an oriented link diagram with at least one crossing and ordered link components, 
and $D_i$ be a sub-diagram of $D$ representing $i$-th link component, $i=1, \cdots , l$. 
We fix a sub-diagram $D_i$. 
We take a componentwise Alexander numbering associated with $D_i$. 
Let $q$ be a crossing of $D$ 
and we denote the four corners touching $q$ by $C_q^1, C_q^2, C_q^3, C_q^4$ clockwise, 
and the regions on $D$ including $C_q^j$ by $R_q^j$, $j=1,2,3,4$. 
If $q$ is  a crossing of $D_i$, 
the regions $R_q^1, R_q^2, R_q^3, R_q^4$ are assigned integers $r_1, r_2, r_3, r_4$   
with $r_1-r_2+r_3-r_4=0$ by Lemma \ref{lem;AlexindKer}. 
If $q$ is a crossing of an arc of $D_i$ and an arc of other components, 
we may assume that $R_q^1$ and $R_q^2$ are subsets of a region $S$ of the diagram $D_i$, 
and that $R_q^3$ and $R_q^4$ are subsets of a region $S'$ of the diagram $D_i$. 
Then the regions $R_q^1, R_q^2, R_q^3, R_q^4$ are assigned integers $r_1, r_1, r_3, r_3$   
with $r_1-r_3=\pm 1$, 
and we have $r_1-r_1+r_3-r_3=0$.  
If $q$ is a crossing not included in $D_i$, 
the regions $R_q^1, R_q^2, R_q^3, R_q^4$ are subsets of a region of the diagram $D_i$. 
Then they are assigned same integer $r_1$,  and we have $r_1-r_1+r_1-r_1=0$. 

Therefore the componentwise Alexander numbering associated with $D_i$ becomes 
a kernel solution for the alternating integral region choice problem of double counting rule. 
\end{proof}

We can obtain kernel solutions for the definite region choice matrix 
from kernel solutions for the alternating region choice matrix and a fixed checkerboard coloring. 

\begin{lemma}\label{lem;KerAtoD}
For a given link diagram, 
we fix a checkerboard coloring. 
We take a kernel solution for 
an alternating region choice matrix of the double counting rule. 
For each region $R$, let $c_R$ and $u_R$ be the integers assigned by 
the checkerboard coloring and the kernel solution respectively. 
Assigning the integer 
$\displaystyle (-1)^{c_R}u_R$ to each region $R$, 
we obtain a kernel solution for 
a definite region choice matrix of the double counting rule. 
\end{lemma}

\begin{proof}
For a crossing $x$ of the diagram, 
We denote the four corners touching $x$ by $C_1, C_2, C_3, C_4$ clockwise, 
and the regions including $C_j$ by $R_j$, $j=1,2,3,4$. 
Then we have $\pm (u_{R_1}-u_{R_2}+u_{R_3}-u_{R_4})=0$. 
We may assume $c_{R_1}=0$. 
Then the equalities $c_{R_2}=1, c_{R_3}=0, c_{R_4}=1$ hold. 
Hence we have 
\begin{eqnarray*}
& & (-1)^{c_{R_1}}u_{R_1} +(-1)^{c_{R_2}}u_{R_2} +(-1)^{c_{R_3}}u_{R_3} +(-1)^{c_{R_4}}u_{R_4} \\
&=& u_{R_1}-u_{R_2}+u_{R_3}-u_{R_4} \\
&=& 0. 
\end{eqnarray*}
\end{proof}

\begin{lemma}\label{lem;KerAtoDsingle}
For a given link diagram, 
we fix a checkerboard coloring. 
We take a kernel solution for 
an alternating region choice matrix of the single counting rule. 
For each region $R$, let $c_R$ and $u_R$ be the integers assigned by 
the checkerboard coloring and the kernel solution respectively. 
Assigning the integer 
$\displaystyle (-1)^{c_R}u_R$ to each region $R$, 
we obtain a kernel solution for 
a definite region choice matrix of the single counting rule. 
\end{lemma}

\begin{proof}
For a crossing $x$ of the diagram, 
We denote the four corners touching $x$ by $C_1, C_2, C_3, C_4$ clockwise, 
and the regions including $C_j$ by $R_j$, $j=1,2,3,4$. 
We may assume $c_{R_1}=0$. 
Then the equalities $c_{R_2}=1, c_{R_3}=0, c_{R_4}=1$ hold. 

If $x$ is not reducible, 
then $R_j$'s are different each other and we have $\pm (u_{R_1}-u_{R_2}+u_{R_3}-u_{R_4})=0$. 
Hence we have 
\begin{eqnarray*}
& & (-1)^{c_{R_1}}u_{R_1} +(-1)^{c_{R_2}}u_{R_2} +(-1)^{c_{R_3}}u_{R_3} +(-1)^{c_{R_4}}u_{R_4} \\
&=& u_{R_1}-u_{R_2}+u_{R_3}-u_{R_4} \\
&=& 0. 
\end{eqnarray*}

We suppose that $x$ is reducible. 
Then there exists just one pair of $R_j$'s coinciding each other. 
If $R_1$ coincides with $R_3$, the equality $\pm (u_{R_1}-u_{R_2}-u_{R_4})=0$ holds. 
Hence we have 
\[ 
(-1)^{c_{R_1}}u_{R_1} +(-1)^{c_{R_2}}u_{R_2} +(-1)^{c_{R_4}}u_{R_4}= u_{R_1}-u_{R_2}-u_{R_4}=0. 
\] 
Otherwise, $R_2$ coincides with $R_4$ and  we have $\pm (u_{R_1}-u_{R_2}+u_{R_3})=0$. 
Hence we have 
\[
(-1)^{c_{R_1}}u_{R_1} +(-1)^{c_{R_2}}u_{R_2} +(-1)^{c_{R_3}}u_{R_3} =u_{R_1}-u_{R_2}+u_{R_3}=0. 
\]
\end{proof}

Similarly, 
we can obtain kernel solutions for the alternating region choice matrix 
from kernel solutions for the definite region choice matrix and a fixed checkerboard coloring. 

\begin{lemma}\label{lem;KerDtoA}
For a given link diagram, 
we fix a checkerboard coloring. 
We take a kernel solution for 
a definite region choice matrix of the double (resp. single) counting rule. 
For each region $R$, let $c_R$ and $u_R$ be the integers assigned by 
the checkerboard coloring and the kernel solution respectively. 
Assigning the integer 
$\displaystyle (-1)^{c_R}u_R$ to each region $R$, 
we obtain a kernel solution for 
an alternating region choice matrix of the double (resp. single) counting rule. 
\hfill $\square$
\end{lemma}

\section{Solutions of the alternating integral region choice problem on knot diagrams}\label{sect;AZRCPpf}


In this section, we give an alternative proof of Theorem \ref{thm;harada;mat}.

First, we observe the alternating integral region choice problem of the double counting rule.

\begin{lemma}
\label{lem;arcfixAZRCker}
Let $D$ be a link diagram with 
$n$ crossings, $n\geq 1$. 
We fix an arc $\gamma$ in the link diagram $D$,
and let $R$ and $R'$ be two regions which are the both sides of the arc $\gamma$. 
Then there exists a kernel solution $\mathbf{u}$ for $A_{a2}(D)$ 
such that the components of $\mathbf{u}$ corresponding to $R$ and $R'$ are $0$ and $1$ respectively.  
\end{lemma}

\begin{proof}
By Lemma \ref{lem;AlexindKer} or \ref{lem;compAlex}, 
there exists a kernel solution $\mathbf{u}$ for $A_{a2}(D)$ 
such that the components of $\mathbf{u}$ corresponding to $R$ and $R'$ are $0$ and $\pm1$ respectively.  
If $R'$ is assigned $-1$, we multiply all components of $\mathbf{u}$ by $-1$. 
\end{proof}

Figure \ref{Fig;arcfixAZRCker} gives an example of 
a link diagram with a kernel solution for an alternating region choice matrix 
such that two regions adjacent to the arc $\gamma$ are assigned $0$ and $1$. 
This kernel solution is obtained from an Alexander numbering. 

\begin{figure}[htbp]
\begin{center}
\includegraphics[height=3.5cm,clip]{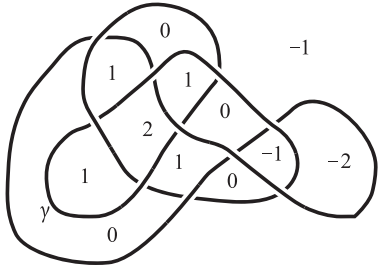}
\end{center}
\caption{A kernel solution for an alternating region choice matrix.}
\label{Fig;arcfixAZRCker}
\end{figure}

\begin{remark}\label{rem;arcfixAZRCker}
In \cite{haradaM}, 
Harada proved Lemma \ref{lem;arcfixAZRCker} for a knot diagram, 
showing that Reidemeister moves and crossing changes preserve 
the existence of the kernel solution, and that 
the knot diagram with only one crossing has a kernel solution. 
His argument is  
similar to 
that due to Ahara and Suzuki \cite{aharasuzuki} for Lemma \ref{lem;arcfixDZRCker}. 
\end{remark}

The following theorem also has been proved by Harada \cite{haradaM} 
for a knot diagram using Lemma \ref{lem;arcfixAZRCker}. 
We give a proof using Lemma \ref{lem;compAlex} instead.

\begin{theorem}
\label{thm;add1AZRC}
Let $D$ be a link diagram with $d$ connected components and $n$ crossings, $n\geq 1$. 
We take a crossing $x$ of $D$ of arcs in same link component. 
There exist $\mathbf{v}_x \in \mathbb{Z}^{n+d+1}$ such that 
any components of $A_{a2}(D)\mathbf{v}_x$ are $0$ 
but the component of $A_{a2}(D)\mathbf{v}_x$ to $x$ is $1$. 
\end{theorem}

\begin{figure}[htbp]
\begin{center}
\includegraphics[height=8cm,clip]{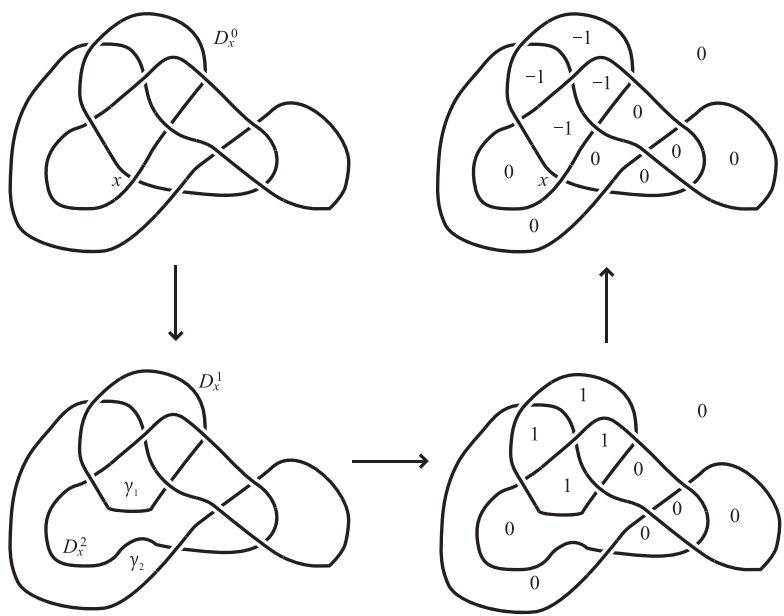}
\end{center}
\caption{Finding 
$\mathbf{v}_x 
$ such that 
any components of $A_{a2}(D)\mathbf{v}_x$ are $0$ 
but the component of $A_{a2}(D)\mathbf{v}_x$ to $x$ is $1$. 
}
\label{Fig;add1AZRC}
\end{figure}

\begin{proof}
The argument is similar to that for knot diagrams due to Harada \cite{haradaM}. 
We orient $D$ arbitrarily. 
We splice $D$ at $x$. 
On Figure \ref{Fig;add1AZRC}, 
this splicing is illustrated as the transformation from top left to bottom left. 
The sub-diagram $D_x^0$ of the link component including  $x$  splits to 
the diagrams of two link components 
$D_x^1$ and $D_x^2$. 
We note $D_x^1$ and $D_x^2$ may intersects each other as link projections. 
Let $\gamma_i$ be an oriented arc in $D_x^i$ appearing after the splice at $x$ for each $i=1,2$. 
We may assume that $\gamma _1$ lies on the left of $\gamma _2$. 
For the diagram $(D\setminus D_x^0)\cup D_x^1 \cup D_x^2$, 
we take the componentwise Alexander numbering associated with $D_x^1$ 
such that the right and left regions of $\gamma _1$ are assigned $0$ and $1$ respectively. 
We denote this assignment of the indexes by $\mathbf{u}'$. 
On Figure \ref{Fig;add1AZRC}, 
$\mathbf{u}'$ is illustrated on bottom right. 
%
%
By Lemma \ref{lem;compAlex}, 
$\mathbf{u}'$ gives a kernel solution  of $A_{a2}((D\setminus D_x^0)\cup D_x^1 \cup D_x^2)$ 
if the spliced diagram has at least one crossing. 
We unsplice $(D\setminus D_x^0)\cup D_x^1 \cup D_x^2$ to $D$ at $x$. 
Let $\varepsilon =1$ if $x$ is a positive crossing, otherwise $\varepsilon =-1$.  
We assign the same integers to all regions of $D$ as the components of $\varepsilon \mathbf{u}'$, 
where the integer assigned to the region between $\gamma _1$ and $\gamma _2$ 
is assigned to the two regions splitting at $x$. 
On Figure \ref{Fig;add1AZRC}, 
this unsplicing  is illustrated as the transformation from bottom right to top right, 
where the crossing $x$ is negative and we have $\varepsilon =-1$. 
Then we obtain the desired $\mathbf{v}_x \in \mathbb{Z}^{n+d+1}$. 
\end{proof}

Theorem \ref{thm;add1AZRC} implies 
Theorem \ref{thm;harada;mat} (2), that is 
the existence of a solution of 
an alternating  integral region choice problem of the double counting rule 
for a knot diagram, 
by the same argument as that  due to Harada \cite{haradaM}. 

\begin{proof}[Proof of Theorem \ref{thm;harada;mat} (2)]
Applying Theorem \ref{thm;add1AZRC} for each crossing $x_i$,  
there exist $\mathbf{v}_i \in \mathbb{Z}^{n+d+1}$ such that 
any components of $A_{a2}(D)\mathbf{v}_i$ are $0$ 
but the $i$-th component of $A_{a2}(D)\mathbf{v}_i$ 
is $1$, $i=1,2, \cdots , n$. 
Let $c_i$ be the $i$-th component of $\mathbf{c}$. 
If we take $\displaystyle \mathbf{u}=-\sum _{i=1}^n c_i\mathbf{v}_i$, 
then we have $A_{a2}(D)\mathbf{u}+\mathbf{c}=\mathbf{0}$. 
\end{proof}

Next, we observe the alternating integral region choice problem of the single counting rule. 
The following lemma has been proved by Harada \cite{haradaM} 
for knot diagrams. 

\begin{lemma}
\label{lem;arcfixAZRCkerSingle}
Let $D$ be a link diagram with 
$n$ crossings, $n\geq 1$. 
We fix an arc $\gamma$ in the link diagram $D$,
and let $R$ and $R'$ be two regions which are the both sides of the arc $\gamma$. 
We take two arbitrary integers $a$ and $b$. 
Then there exists a kernel solution $\mathbf{u}$ for $A_{a1}(D)$ 
such that the components of $\mathbf{u}$ corresponding to $R$ and $R'$ are $a$ and $b$ respectively.  
\end{lemma}

\begin{proof}
The argument is same as that for knot diagrams due to Harada \cite{haradaM}. 
His argument is  
similar to that due to Ahara and Suzuki \cite{aharasuzuki} for Lemma \ref{lem;arcfixDZRCkerSingle}. 

We use an induction on the number of reducible crossings. 

If the given link diagram $D$ is irreducible, 
the matrices $A_{a1}(D)$ and $A_{a2}(D)$ coincide.  
We apply Lemma \ref{lem;arcfixAZRCker} to the pairs $R, R'$ and $R', R$ 
in order to a kernel solutions $\mathbf{u}'$ and $\mathbf{u}''$ respectively. 
Then the components of $\mathbf{u}'$ corresponding to $R$ and $R'$ are $0$ and $1$ respectively,   
and the components of $\mathbf{u}''$ corresponding to $R$ and $R'$ are $1$ and $0$ respectively.  
Therefore $\mathbf{u}=a\mathbf{u}''+b\mathbf{u}'$ is the desired kernel solution for $A_{a1}(D)$ 
on the irreducible diagram $D$. 

\begin{figure}[htbp]
\begin{center}
\includegraphics[height=6cm,clip]{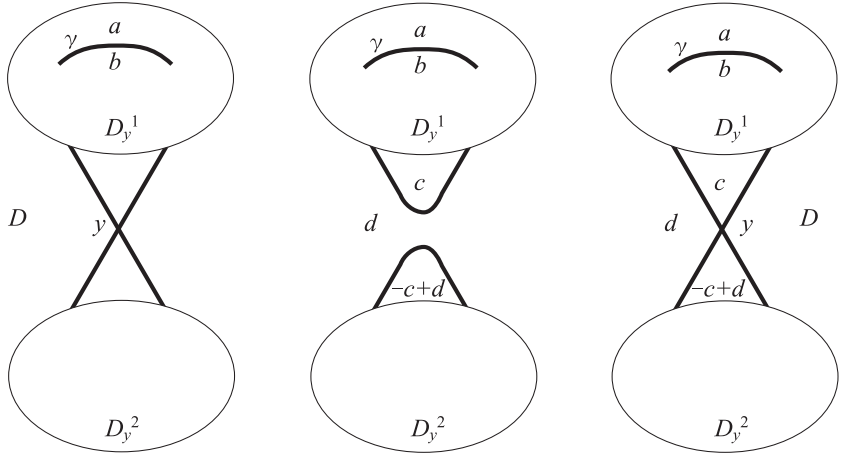}
\end{center}
\caption{Splicing at $y$ and obtaining a kernel solution.}
\label{Fig;arcfixAZRCkersingle}
\end{figure}

We assume that there exists a deseired kernel solution
if the number of reducible crossings is less than $k$. 
We suppose that the link diagram $D$ has $k$ reducible crossings. 
We take a reducible crossing $y$, 
and orient $D$ arbitrarily. 
We splice $D$ at $y$. 
The diagram $D$ splits to 
the disjoint link diagrams 
$D_y^1$ and $D_y^2$. 
Each of them has less reducible crossings than $k$. 
We may assume that the given arc $\gamma$ lies on $D_y^1$. 
On Figure \ref{Fig;arcfixAZRCkersingle}, 
we obtain the middle diagram 
splicing the reducible crossing $y$ of the left diagram, 
where we omit over/under information for $y$ and the orientation of $D$. 
Let $\gamma_i$ be an arc in $D_y^i$ appearing after the splice at $y$ for each $i=1,2$.  
We denote the region between the arcs $\gamma _1$ and $\gamma _2$ by $R_y^0$, 
and another region adjacent to $\gamma _j$ by $R_y^j$, $j=1,2$. 
We ignore $D_y^2$. 
If $D_y^1$ has no crossing, 
we assign $a$ and $b$ to the regions including $R$ and $R'$ respectively, 
and $0$ to other regions of $D_y^1$ 
where we may assign arbitrary integers. 
If $D_y^1$ has at least one crossing, 
we apply the assumption of induction to $D_y^1$ and $\gamma$. 
Then we  obtain a kernel solution for $A_{a1}(D_y^1)$  
whose components corresponding to the regions including $R$ and $R'$ are $a$ and $b$ respectively.  
Let $c\in \mathbb{Z}$ assigned to $R_y^1$ 
and $d \in \mathbb{Z}$ to the region of $D_y^1$ including $R_y^0$. 
On the middle of Figure \ref{Fig;arcfixAZRCkersingle}, we write $c$ and $d$, 
though we omit $\gamma _i$ and $R_y^j$. 
We assign the integer $-c+d$ to the region $R_y^2$ on $D_y^1\cup D_y^2$ as the middle of Figure \ref{Fig;arcfixAZRCkersingle}. 
We ignore $D_y^1$.  
If the diagram $D_y^2$ has no crossing, 
we assign  $0$ 
to the regions of $D_y^2$ 
including  neither $R_y^0$ nor $R_y^2$, 
though we may assign arbitrary integers. 
Otherwise we apply the assumption of  the induction to $D_y^2$ and $\gamma _2$,  
then we obtain a kernel solution for $A_{a1}(D_y^2)$ 
whose components corresponding 
to $R_y^2$ and the region including $R_y^0$ are $-c+d$ and $d$ respectively.  
Let $\tilde{R}$ be a region of the diagram $D_y^1 \cup D_y^2$. 
If $\tilde{R}$ is $R_y^0$, we assign $d$ to $\tilde{R}=R_y^0$. 
Otherwise $\tilde{R}$ coincides with one of regions of $D_y^1$ or $D_y^2$, 
then we assign to $\tilde{R}$ same integer as the region of  $D_y^1$ or $D_y^2$. 
Therefore we obtain a kernel solution $\mathbf{u}'$ of $A_{a1}(D_y^1 \cup D_y^2)$. 
We unsplice $D_y^1 \cup D_y^2$ at $y$. 
We assign the same integer as either $\mathbf{u}'$ to all regions of $D$, 
where 
the region touching  $y$ twice is assigned $d$, 
as illustrated on the right of Figure \ref{Fig;arcfixAZRCkersingle}. 
Then we obtain the desired kernel solution for $A_{a1}(D)$ since we have $c-d+(-c+d)=0$. 
\end{proof}

The following lemma also has been proved by Harada \cite{haradaM} for knot diagrams. 
He proved it as a corollary to Lemma \ref{lem;arcfixAZRCkerSingle}: 
the region $R_y^2$ in the proof of  Lemma \ref{lem;arcfixAZRCkerSingle} 
is assigned $-c+d+\varepsilon$  
instead of $-c+d$, where $\varepsilon =1$ if $y$ is positive, otherwise $\varepsilon =-1$. 
We give an alternative proof. 

\begin{lemma}
\label{lem;add1AZRCred}
Let $D$ be a link diagram with $d$ connected components and $n$ crossings, $n\geq 1$. 
Let $y$ be a reducible crossing of $D$. 
There exist $\mathbf{v}_y \in \mathbb{Z}^{n+d+1}$ such that 
any components of $A_{a1}(D)\mathbf{v}_y$ are $0$ 
but the component of $A_{a1}(D)\mathbf{v}_y$ to $y$ is $1$. 
\end{lemma}

\begin{figure}[htbp]
\begin{center}
\includegraphics[height=6cm,clip]{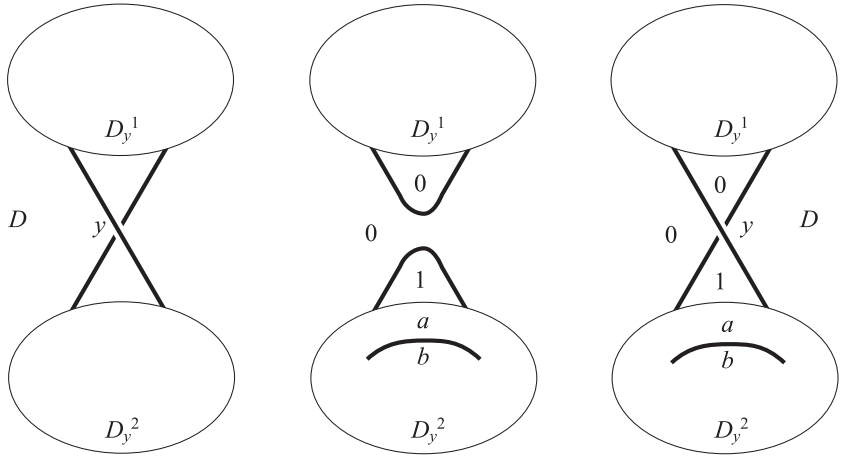}
\end{center}
\caption{Obtaining $\mathbf{v}_y$ for the reducible and positive crossing $y$.}
\label{Fig;add1AZRCred}
\end{figure}

\begin{figure}[htbp]
\begin{center}
\includegraphics[height=6cm,clip]{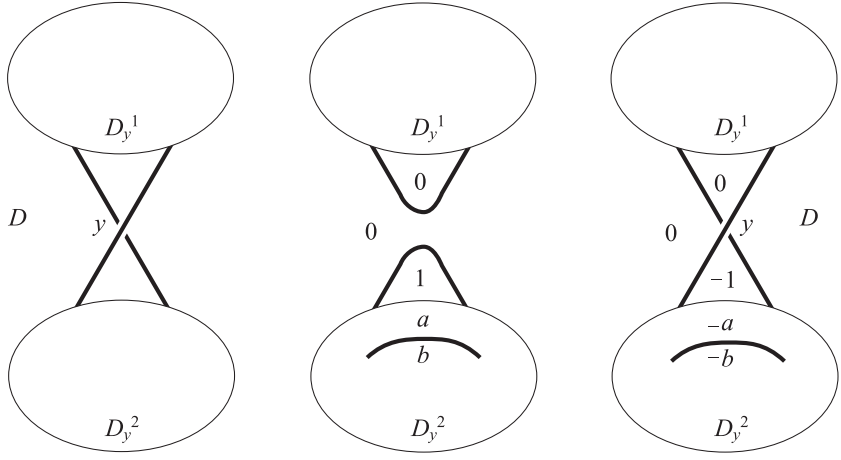}
\end{center}
\caption{Obtaining $\mathbf{v}_y$ for the reducible and negative crossing $y$.}
\label{Fig;add1AZRCred2}
\end{figure}

\begin{proof}
We orient $D$ arbitrarily. 
We splice $D$ at the reducible crossing $y$. 
The diagram $D$ splits to 
the two disjoint link diagrams 
$D_y^1$ and $D_y^2$. 
Let $\gamma_i$ be an arc in $D_y^i$ appearing after the splice at $y$ for each $i=1,2$.  
We denote the region between the arcs $\gamma _1$ and $\gamma _2$ by $R_y^0$, 
and another region adjacent to $\gamma _j$ by $R_y^j$, $j=1,2$. 
We assign integers $0, 0, 1$ to $R_y^0, R_y^1, R_y^2$ respectively,  
as illustrated on the middle of Figure \ref{Fig;add1AZRCred} or \ref{Fig;add1AZRCred2}. 
If $D_y^2$ has at least one crossing, 
we apply Lemma \ref{lem;arcfixAZRCkerSingle} to 
$D_y^2$ and $\gamma _2$, 
in order to obtain a kernel solution for 
$A_{a1}(D_y^2)$ such that 
the region including $R_y^0$ is assigned $0$ and $R_y^2$ is assigned $1$. 
If $D_y^2$ has no crossing, 
we assign  $0$ 
to the regions of $D_y^2$ 
including  neither $R_y^0$ nor $R_y^2$, 
though we may assign arbitrary integers. 
We assign $0$ to all regions of $D_y^1$: 
it gives the trivial kernel solution for $A_{a1}(D_y^1)$ if $D_y^1$ has at least one crossing. 
Each of the other regions of $D_y^1 \cup D_y^2$ than  $R_y^0, R_y^1, R_y^2$ 
is a region of $D_y^1$ or $D_y^2$, 
then it has been assigned an integer. 
Therefore we obtain a kernel solution $\mathbf{u}'$ for $A_{a1}(D_y^1 \cup D_y^2)$ 
such that the components of $\mathbf{u}'$ corresponding to $R_y^0, R_y^1, R_y^2$ are  $0, 0, 1$ respectively. 
Let $\varepsilon =1$ if $y$ is a positive crossing, otherwise $\varepsilon =-1$.  
We unsplice at $y$ 
and assign the same integers to all regions of $D$ 
as $\varepsilon \mathbf{u}'$, 
where the region touching $y$ twice is assigned $0$. 
On the right of Figure \ref{Fig;add1AZRCred}, 
the crossing $y$ is positive 
and regions of $D$ are assigned the components of $\mathbf{u}'$. 
On the right of Figure \ref{Fig;add1AZRCred2}, 
the crossing $y$ is negative 
and regions of $D$ are assigned the components of $-\mathbf{u}'$. 
Then we obtain the desired $\mathbf{v}_y \in \mathbb{Z}^{n+d+1}$. 
\end{proof}

Combining Lemma \ref{lem;add1AZRCred} with 
Theorem \ref{thm;harada;mat} (2), 
we obtain the proof of 
Theorem \ref{thm;harada;mat} (1), that is 
the existence of a solution of 
an alternating  integral region choice problem of the single counting rule 
for a knot diagram, 
by the same argument as that  due to Harada \cite{haradaM}. 

\begin{proof}[Proof of Theorem \ref{thm;harada;mat} (1)]
If the given knot diagram $D$ is irreducible, 
then we have $A_{a1}(D)=A_{a2}(D)$ and 
a solution of the double counting rule, which has been obtained,  is also 
a solution of the single counting rule. 

We suppose that the knot diagram $D$ has at least one reducible crossing. 
For a region $R_j$, 
let $\mathcal{X}_j$ be the set of reducible crossings touched by $R_j$ twice, $j=1, \cdots , n+2$. 
A set $\mathcal{X}_j$ might be empty. 
Applying Lemma  \ref{lem;add1AZRCred} 
for each reducible crossing $y\in \mathcal{X}_j$, 
we obtain $\mathbf{v}_y \in \mathbb{Z}^{n+2}$ such that 
any components of $A_{a1}(D)\mathbf{v}_y$ are $0$ 
but the component of $A_{a1}(D)\mathbf{v}_y$ to $y$ is $1$. 
We take $\mathbf{r}_j\in \mathbb{Z}^{n+2}$ such that 
any components of $\mathbf{r}_j$ are $0$ 
but the $j$-th component is $1$. 
Choosing $R_j$ once corresponds with $\mathbf{r}_j$. 
By the definitions of the alternating region choice matrices, 
we have 
\[
 A_{a2}(D)\mathbf{r}_j - A_{a1}(D)\mathbf{r}_j =\sum _{y\in \mathcal{X}_j} A_{a1}(D)\mathbf{v}_y.
\] 

Applying Theorem \ref{thm;harada;mat} (2), which has been proved, 
we obtain a solution of double counting rule, 
$\mathbf{w} \in \mathbb{Z}^{n+2}$ with $A_{a2}(D)\mathbf{w}+\mathbf{c}=\mathbf{0}$. 
Let $w_j$ be the $j$-th component of $\mathbf{w}$, $j=1, \cdots , n+2$. 
We note $\displaystyle \mathbf{w}=\sum _j w_j\mathbf{r}_j$. 
We take $\displaystyle \mathbf{u}=\mathbf{w}+\sum _j w_j \sum _{y\in \mathcal{X}_j} \mathbf{v}_y$. 
Then we have 
\begin{eqnarray*}
A_{a1}(D)\mathbf{u}
&=& A_{a1}(D)\mathbf{w}+\sum _j w_j \sum _{y\in \mathcal{X}_j} A_{a1}(D)\mathbf{v}_y \\
&=& \sum _j w_j \left( A_{a1}(D)\mathbf{r}_j + \sum _{y\in \mathcal{X}_j} A_{a1}(D)\mathbf{v}_y \right) \\
&=& \sum _j w_j A_{a2}(D)\mathbf{r}_j \\
&=& A_{a2}(D)\mathbf{w} \\
&=& -\mathbf{c}. 
\end{eqnarray*}
Therefore we have $A_{a1}(D)\mathbf{u} +\mathbf{c}=\mathbf{0}$. 
\end{proof}

\begin{remark}
In Section \ref{sect;RCmatrixS}, 
we prove that 
the first and second results of Theorem \ref{thm;harada;mat}  are equivalent. 
\end{remark}

The proof of Lemma \ref{lem;add1AZRCred} for knot diagrams due to Harada \cite{haradaM} 
implies the following fact. 
We give an alternative proof. 

\begin{lemma}\label{lem;arcfixadd1AZRCred}
Let $D$ be a link diagram with $d$ connected components and $n$ crossings, $n\geq 1$. 
Let $y$ be a reducible crossing of $D$. 
We fix an arc $\gamma$ in the link diagram $D$,
and let $R$ and $R'$ be two regions which are the both sides of the arc $\gamma$. 
We take two arbitrary integers $a$ and $b$. 
There exist $\mathbf{v}_y \in \mathbb{Z}^{n+d+1}$ such that 
any components of $A_{a1}(D)\mathbf{v}_y$ are $0$ 
but the component of $A_{a1}(D)\mathbf{v}_y$ to $y$ is $1$,  
and such that the components of $\mathbf{v}_y$ corresponding to $R$ and $R'$ are $a$ and $b$ respectively.  
\end{lemma}

\begin{proof}
Applying Lemma \ref{lem;add1AZRCred}, 
we obtain $\mathbf{v}'_y \in \mathbb{Z}^{n+d+1}$ such that 
any components of $A_{a1}(D)\mathbf{v}'_y$ are $0$ 
but the component of $A_{a1}(D)\mathbf{v}'_y$ to $y$ is $1$. 
We denote 
the components of $\mathbf{v}'_y$ corresponding to $R$ and $R'$ 
by $a'$ and $b'$ respectively. 
Applying  Lemma \ref{lem;arcfixAZRCkerSingle}, 
there exists a kernel solution $\mathbf{u}$ for $A_{a1}(D)$ 
such that the components of $\mathbf{u}$ corresponding to $R$ and $R'$ 
are $a-a'$ and $b-b'$ respectively.  
Let $\mathbf{v}_y=\mathbf{u}+\mathbf{v}'_y$. 
Then we obtain the desired $\mathbf{v}_y$.
\end{proof}

We note that 
we use Lemma \ref{lem;add1AZRCred} but does not use Lemma \ref{lem;arcfixadd1AZRCred} 
to prove  Theorem \ref{thm;harada;mat} (1) in this article. 

\begin{remark}
In \cite{shimizu}, 
Shimizu took checkerboard colorings to show that 
a region crossing change is an unknotting operation on a knot diagram. 
In the above argument for Theorem \ref{thm;harada;mat}, 
we take Alexander numberings instead of checkerboard colorings. 
Then an extension of her argument is given. 
\end{remark}

\section{Solutions of the definite integral region choice problem on knot diagrams}\label{sect;DZRCPpf}


In this section, we give an alternative proof of Theorem \ref{thm;aharasuzuki;mat}.

First, we observe the definite integral region choice problem of the double counting rule.

\begin{lemma}
\label{lem;arcfixDZRCker}
Let $D$ be a link diagram or projection with 
$n$ crossings, $n\geq 1$. 
We fix an arc $\gamma$ in the link diagram $D$,
and let $R$ and $R'$ be two regions which are the both sides of the arc $\gamma$. 
Then there exists a kernel solution for $A_{d2}(D)$ 
such that the components of $\mathbf{u}$ corresponding to $R$ and $R'$ are $0$ and $1$ respectively.  
\end{lemma}

\begin{proof}
If $D$ is a link projection, we arbitrarily add over/under information to crossings. 
Then we may assume that $D$ is a link diagram. 
By Lemma \ref{lem;arcfixAZRCker}, 
there exists a kernel solution $\tilde{\mathbf{u}}$ for $A_{a2}(D)$ 
such that the components of $\tilde{\mathbf{u}}$ corresponding to $R$ and $R'$ are $0$ and $1$ respectively.  
We fix a checkerboard coloring such that the region $R'$ is assigned $0$.
Applying Lemma \ref{lem;KerAtoD} to $\tilde{\mathbf{u}}$,  
we obtain a kernel solution $\mathbf{u}$ for $A_{d2}(D)$ 
such that the components of $\mathbf{u}$ corresponding to $R$ and $R'$ are $0$ and $1$ respectively.  
\end{proof}

Figure \ref{Fig;arcfixDZRCker} gives an example of 
a link diagram with a kernel solution for a definite region choice matrix 
such that two regions adjacent to the arc $\gamma$ are assigned $0$ and $1$. 
This kernel solution is obtained from Figure \ref{Fig;arcfixAZRCker} applying Lemma \ref{lem;KerAtoD}. 

\begin{figure}[htbp]
\begin{center}
\includegraphics[height=3.5cm,clip]{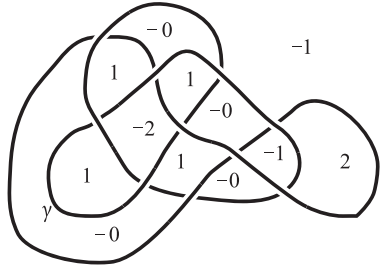}
\end{center}
\caption{A kernel solution for a definite region choice matrix.}
\label{Fig;arcfixDZRCker}
\end{figure}

\begin{remark}\label{rem;arcfixDZRCker}
In \cite{aharasuzuki}, 
Ahara and Suzuki proved Lemma \ref{lem;arcfixDZRCker} for a knot diagram, 
showing that Reidemeister moves preserve 
the existence of the kernel solution, and that 
the knot diagram with only one crossing has a kernel solution. 
\end{remark}

The following theorem also has been proved by Ahara and Suzuki \cite{aharasuzuki} 
for a knot diagram splicing at the given crossing and applying Lemma \ref{lem;arcfixDZRCker}. 
We give an alternative proof below. 

\begin{theorem}
\label{thm;add1DZRC}
Let $D$ be a link diagram or projection with $d$ connected components and $n$ crossings, $n\geq 1$. 
We take a crossing $x$ of $D$ of arcs in same link component. 
There exist $\mathbf{v}_x \in \mathbb{Z}^{n+d+1}$ such that 
any components of $A_{d2}(D)\mathbf{v}_x$ are $0$ 
but the component of $A_{d2}(D)\mathbf{v}_x$ to $x$ is $1$. 
\end{theorem}

\begin{proof}
If $D$ is a link projection, we arbitrarily add over/under information to crossings. 
Then we may assume that $D$ is a link diagram. 
By Theorem \ref{thm;add1AZRC}, 
there exist $\mathbf{w}_x \in \mathbb{Z}^{n+d+1}$ such that 
any components of $A_{a2}(D)\mathbf{w}_x$ are $0$ 
but the component of $A_{a2}(D)\mathbf{w}_x$ to $x$ is $1$. 
For each region $R$, 
let $w_R$ be the component of $\mathbf{w}_x$ to $R$.  
We denote the four corners touching $x$ by $C_1, C_2, C_3, C_4$ clockwise, 
and the regions including $C_j$ by $R_j$, $j=1,2,3,4$. 
Then we have $ w_{R_1}-w_{R_2}+w_{R_3}-w_{R_4}=1$. 
We fix the checkerboard coloring such that 
the region $R_1$ is assigned $0$. 
For each region $R$, 
let $c_R$ be the integer assigned by this checkerboard coloring. 
Let $\mathbf{v}_x$ be the vector in $\mathbb{Z}^{n+d+1}$ 
such that 
the component to $R$ is $\displaystyle (-1)^{c_R}w_R$. 
By similar argument to the proof of Lemma \ref{lem;KerAtoD}, 
it is shown that 
any components of $A_{d2}(D)\mathbf{v}_x$ are $0$ 
but the component of $A_{d2}(D)\mathbf{v}_x$ to $x$ is $1$. 
Figure \ref{Fig;add1DZRC} gives an example of the above process 
from $\mathbf{w}_x$ to $\mathbf{v}_x$, 
where $\mathbf{w}_x$ is a solution illustrated on top right of Figure \ref{Fig;add1AZRC}. 
\end{proof}

\begin{figure}[htbp]
\begin{center}
\includegraphics[height=3.5cm,clip]{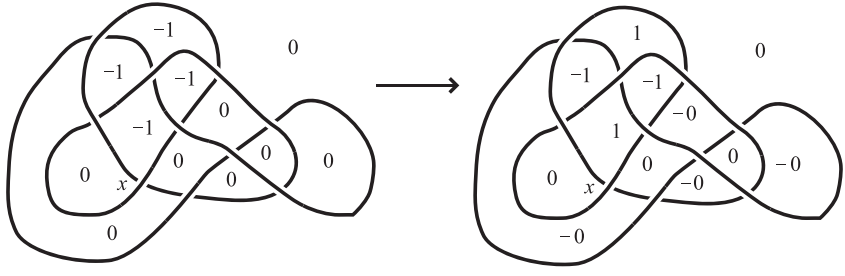}
\end{center}
\caption{Finding 
$\mathbf{v}_x 
$ such that 
any components of $A_{d2}(D)\mathbf{v}_x$ are $0$ 
but the component of $A_{d2}(D)\mathbf{v}_x$ to $x$ is $1$. 
}
\label{Fig;add1DZRC}
\end{figure}

In the above proof, we do not use Lemma  \ref{lem;arcfixDZRCker} immediately.  

Theorem \ref{thm;add1DZRC} implies 
Theorem \ref{thm;aharasuzuki;mat} (2), that is 
the existence of a solution of 
a definite  integral region choice problem of the double counting rule 
for a knot diagram, 
by the same argument as that  
due to Ahara and Suzuki \cite{aharasuzuki}. 

\begin{proof}[Proof of Theorem \ref{thm;aharasuzuki;mat} (2)]
Applying Theorem \ref{thm;add1DZRC} for each crossing $x_i$,  
there exist $\mathbf{v}_i \in \mathbb{Z}^{n+d+1}$ such that 
any components of $A_{d2}(D)\mathbf{v}_i$ are $0$ 
but the $i$-th component of $A_{a2}(D)\mathbf{v}_i$ is $1$, $i=1,2, \cdots , n$. 
Let $c_i$ be the $i$-th component of $\mathbf{c}$. 
If we take $\displaystyle \mathbf{u}=-\sum _{i=1}^n c_i\mathbf{v}_i$, 
then we have $A_{d2}(D)\mathbf{u}+\mathbf{c}=\mathbf{0}$. 
\end{proof}

Next, we observe the definite integral region choice problem of the single counting rule. 
The following lemma has been proved by Ahara and Suzuki \cite{aharasuzuki} 
for knot diagrams. 
We give an alternative proof. 

\begin{lemma}
\label{lem;arcfixDZRCkerSingle}
Let $D$ be a link diagram or projection with 
$n$ crossings, $n\geq 1$.
We fix an arc $\gamma$ in the link diagram $D$,
and let $R$ and $R'$ be two regions which are the both sides of the arc $\gamma$. 
We take two arbitrary integers $a$ and $b$. 
Then there exists a kernel solution $\mathbf{u}$ for $A_{d1}(D)$ 
such that the components of $\mathbf{u}$ corresponding to $R$ and $R'$ are $a$ and $b$ respectively.  
\end{lemma}

\begin{proof}
If $D$ is a link projection, we arbitrarily add over/under information to crossings. 
Then we may assume that $D$ is a link diagram. 
We fix a checkerboard coloring of $D$. 
Let $c_R$ and $c_{R'}$ be the assigned integers to $R$ and $R'$ respectively by the fixed checkerboard coloring. 
By Lemma \ref{lem;arcfixAZRCkerSingle}, 
there exists a kernel solution $\mathbf{w}$ for $A_{a1}(D)$ 
such that the components of $\mathbf{w}$ corresponding to $R$ and $R'$ 
are $(-1)^{c_R}a$ and $(-1)^{c_{R'}}b$ respectively.  
We apply Lemma \ref{lem;KerAtoDsingle} to $\mathbf{w}$. 
Then we obtain a kernel solution $\mathbf{u}$ for $A_{d1}(D)$ 
such that the components of $\mathbf{u}$ corresponding to $R$ and $R'$ are $a$ and $b$ respectively.  
\end{proof}

The following lemma also has been proved by Ahara and Suzuki \cite{aharasuzuki} for knot diagrams. 
They  proved it as a corollary to Lemma \ref{lem;arcfixDZRCkerSingle}. 
We give an alternative proof. 

\begin{lemma}
\label{lem;add1DZRCred}
Let $D$ be a link diagram or projection 
with $d$ connected components and $n$ crossings, $n\geq 1$. 
Let $y$ be a reducible crossing of $D$. 
There exist $\mathbf{v}_y \in \mathbb{Z}^{n+d+1}$ such that 
any components of $A_{d1}(D)\mathbf{v}_y$ are $0$ 
but the component of $A_{d1}(D)\mathbf{v}_y$ to $y$ is $1$. 
\end{lemma}

\begin{proof}
If $D$ is a link projection, we arbitrarily add over/under information to crossings. 
Then we may assume that $D$ is a link diagram. 
By Lemma \ref{lem;add1AZRCred}, 
there exist $\mathbf{w}_y \in \mathbb{Z}^{n+d+1}$ such that 
any components of $A_{a1}(D)\mathbf{w}_y$ are $0$ 
but the component of $A_{a1}(D)\mathbf{w}_y$ to $y$ is $1$. 
For each region $R$, 
let $w_R$ be the componeto of $\mathbf{w}_y$ to $R$.  
We denote the four corners touching $x$ by $C_1, C_2, C_3, C_4$ clockwise, 
and the regions including $C_j$ by $R_j$, $j=1,2,3,4$. 
We may asuume that $R_3$ coincides with $R_1$.
Then we have $ w_{R_1}-w_{R_2}-w_{R_4}=\pm 1$. 
If  $w_{R_1}-w_{R_2}-w_{R_4}=1$, that is the crossing $y$ is negative, 
we fix the checkerboard coloring such that 
the region $R_1$ is assigned $0$ . 
Otherwise, we fix the checkerboard coloring such that 
the region $R_1$ is assigned $1$ . 
For each region $R$, 
let $c_R$ be the integer assigned by the fixed checkerboard coloring. 
Let $\mathbf{v}_y$ be the vector in $\mathbb{Z}^{n+d+1}$ 
such that 
the component to $R$ is $\displaystyle (-1)^{c_R}w_R$. 
By similar argument to the proof of Lemma \ref{lem;KerAtoDsingle}, 
it is shown that 
any components of $A_{d1}(D)\mathbf{v}_y$ are $0$ 
but the component of $A_{d1}(D)\mathbf{v}_y$ to $y$ is $1$. 
\end{proof}

Combining Lemma \ref{lem;add1DZRCred} with 
Theorem \ref{thm;aharasuzuki;mat} (2), 
we obtain the proof of 
Theorem \ref{thm;aharasuzuki;mat} (1), that is 
the existence of a solution of 
a definite  integral region choice problem of the single counting rule 
for a knot diagram, 
by the same argument as that 
due to Ahara and Suzuki \cite{aharasuzuki}. 

\begin{proof}[Proof of Theorem \ref{thm;aharasuzuki;mat} (1)]
If the given knot diagram or projection $D$ is irreducible, 
then we have $A_{d1}(D)=A_{d2}(D)$ and 
a solution of the double counting rule, which has been obtained,  is also 
a solution of the single counting rule. 

We suppose that the knot diagram or projection $D$ has at least one reducible crossing. 
For a region $R_j$, 
let $\mathcal{X}_j$ be the set of reducible crossings touched by $R_j$ twice, $j=1, \cdots , n+2$. 
A set $\mathcal{X}_j$ might be empty. 
Applying Lemma \ref{lem;add1DZRCred} 
for each reducible crossing $y\in \mathcal{X}_j$, 
we obtain $\mathbf{v}_y \in \mathbb{Z}^{n+2}$ such that 
any components of $A_{d1}(D)\mathbf{v}_y$ are $0$ 
but the component of $A_{d1}(D)\mathbf{v}_y$ to $y$ is $1$. 
We take $\mathbf{r}_j\in \mathbb{Z}^{n+2}$ such that 
any components of $\mathbf{r}_j$ are $0$ 
but the $j$-th component is $1$. 
Choosing $R_j$ once corresponds with $\mathbf{r}_j$. 
By the definitions of the definite region choice matrices, 
we have 
\[
 A_{d2}(D)\mathbf{r}_j - A_{d1}(D)\mathbf{r}_j =\sum _{y\in \mathcal{X}_j} A_{d1}(D)\mathbf{v}_y.
\] 

Applying Theorem \ref{thm;aharasuzuki;mat}  (2), which has been proved, 
we obtain a solution of double counting rule, 
$\mathbf{w} \in \mathbb{Z}^{n+2}$ with $A_{d2}(D)\mathbf{w}+\mathbf{c}=\mathbf{0}$. 
Let $w_j$ be the $j$-th component of $\mathbf{w}$, $j=1, \cdots , n+2$. 
We note $\displaystyle \mathbf{w}=\sum _j w_j\mathbf{r}_j$. 
We take $\displaystyle \mathbf{u}=\mathbf{w}+\sum _j w_j \sum _{y\in \mathcal{X}_j} \mathbf{v}_y$. 
Then we have 
\begin{eqnarray*}
A_{d1}(D)\mathbf{u}
&=& A_{d1}(D)\mathbf{w}+\sum _j w_j \sum _{y\in \mathcal{X}_j} A_{d1}(D)\mathbf{v}_y \\
&=& \sum _j w_j \left( A_{d1}(D)\mathbf{r}_j + \sum _{y\in \mathcal{X}_j} A_{d1}(D)\mathbf{v}_y \right) \\
&=& \sum _j w_j A_{d2}(D)\mathbf{r}_j \\
&=& A_{d2}(D)\mathbf{w} \\
&=& -\mathbf{c}. 
\end{eqnarray*}
Therefore we have $A_{d1}(D)\mathbf{u} +\mathbf{c}=\mathbf{0}$. 
\end{proof}

\begin{remark}
In Section \ref{sect;RCmatrixS}, 
we prove that 
the first and second results of Theorem \ref{thm;aharasuzuki;mat} are equivalent. 
\end{remark}

The proof of Lemma \ref{lem;add1DZRCred} for knot diagrams due to Ahara and Suzuki \cite{aharasuzuki} 
implies the following fact. 
We give an alternative proof. 

\begin{lemma}\label{lem;arcfixadd1DZRCred}
Let $D$ be a link diagram or projection with $d$ connected components and $n$ crossings, $n\geq 1$. 
Let $y$ be a reducible crossing of $D$. 
We fix an arc $\gamma$ in the link diagram $D$,
and let $R$ and $R'$ be two regions which are the both sides of the arc $\gamma$. 
We take two arbitrary integers $a$ and $b$. 
There exist $\mathbf{v}_y \in \mathbb{Z}^{n+d+1}$ such that 
any components of $A_{d1}(D)\mathbf{v}_y$ are $0$ 
but the component of $A_{d1}(D)\mathbf{v}_y$ to $y$ is $1$,  
and such that the components of $\mathbf{v}_y$ corresponding to $R$ and $R'$ are $a$ and $b$ respectively.  
\end{lemma}

\begin{proof}
Applying Lemma \ref{lem;add1DZRCred}, 
we obtain $\mathbf{v}'_y \in \mathbb{Z}^{n+d+1}$ such that 
any components of $A_{d1}(D)\mathbf{v}'_y$ are $0$ 
but the component of $A_{d1}(D)\mathbf{v}'_y$ to $y$ is $1$. 
We denote 
the components of $\mathbf{v}'_y$ corresponding to $R$ and $R'$ 
by $a'$ and $b'$ respectively. 
Applying  Lemma \ref{lem;arcfixDZRCkerSingle}, 
there exists a kernel solution $\mathbf{u}$ for $A_{d1}(D)$ 
such that the components of $\mathbf{u}$ corresponding to $R$ and $R'$ 
are $a-a'$ and $b-b'$ respectively.  
Let $\mathbf{v}_y=\mathbf{u}+\mathbf{v}'_y$. 
Then we obtain the desired $\mathbf{v}_y$.
\end{proof}

We note that 
we use Lemma \ref{lem;add1DZRCred} but does not use Lemma \ref{lem;arcfixadd1DZRCred} 
to prove  Theorem \ref{thm;aharasuzuki;mat} (1) in this article.

\section{Region choice matrices of the double counting rule}\label{sect;RCmatrix}


From now on, 
we change link projections to link diagrams 
adding over/under information to crossings arbitrarily. 

Applying the arguments in the original proofs of Theorem \ref{thm;aharasuzuki;mat}  and \ref{thm;harada;mat} 
in \cite{aharasuzuki, haradaM} to link diagrams,  
the ranks of the definite and alternating region choice matrices are determined. 
In this section, we show that on the double counting rule. 

\begin{theorem}\label{thm;rank}
Let $D$ be a diagram of an $l$-component link. 
We assume that $D$ has $d$ connected components and $n$ crossings, 
$n\geq 1$.  
Then each rank of the definite and alternating region choice matrices of the double counting rule, 
$A_{d2}(D)$ and $A_{a2}(D)$, 
is $n+d-l$,  
and each 
rank of the $\mathbb{Z}$-submodules 
$\{ \mathbf{u} \in \mathbb{Z}^{n+d+1} \mid A_{d2}(D) \mathbf{u}=\mathbf{0} \}$
and 
$\{ \mathbf{u} \in \mathbb{Z}^{n+d+1} \mid A_{a2}(D) \mathbf{u}=\mathbf{0} \}$
is $l+1$.
\end{theorem}

If we transpose the incidence matrix 
induced by Cheng and Gao \cite{chenggao} and Hashizume \cite{hashizume2013}, 
it is same as the definite region choice matrix of the single counting rule 
up to permutations of rows and columns. 
This transposed matrix also coincides with the alternating region choice matrix of the single counting rule 
modulo $2$ up to permutations of rows and columns. 
For irreducible diagrams, 
Theorem \ref{thm;rank} is an extension of their result on the rank of the incidence matrices. 

It is well known that  
we can make any link diagram into a diagram of a trivial link after some crossing changes, 
and that 
some Reidemeister moves can transform any pair of non-trivial diagrams of a trivial link each other. 
We prove Theorem \ref{thm;rank} 
using a similar argument to that in Appendix A 
of the article written by Ahara and Suzuki \cite{aharasuzuki}.  
That means it is similar to the proof of 
the invariance for the Alexander polynomial in \cite{alexander}. 

\begin{lemma}
\label{lem;ccAZRCmat}
If we change a crossing of a link diagram 
admitting integral region choices,  
the changed diagram also admit it,  
the definite region choice matrix of double counting rule is preserved, 
and  
the row concerning this crossing is multiplied by $-1$ 
for the alternating region choice matrix of the double counting rule. 
\end{lemma}

\begin{proof}
It is cleared by the definitions of region choice matrices.  
\end{proof}

\begin{remark}
In \cite{haradaM}, 
Harada proved that 
crossings changes preserve kernel solutions for 
alternating region choice matrices of the double counting rule 
on any knot diagrams.
Lemma \ref{lem;ccAZRCmat} implies that 
his claim holds on any link diagrams.  
\end{remark}

\begin{lemma}
\label{lem;R1mat}
A Reidemeister move I between link diagrams  admitting integral region choices
preserves 
the rank of the $\mathbb{Z}$-submodules of 
kernel solutions for 
definite and alternating region choice matrices of the double counting rule. 
\end{lemma}

\begin{figure}[htbp]
\begin{center}
\includegraphics[height=1cm,clip]{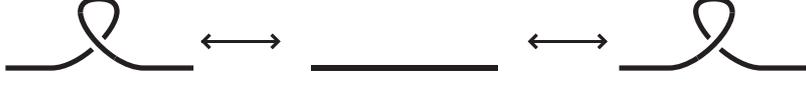}
\end{center}
\caption{Reidemeister moves I among the diagrams $D_+$, $D$, and $D_-$.}
\label{Fig;R1}
\end{figure}

\begin{proof}
(cf. \cite{aharasuzuki})
Let the middle of Figure \ref{Fig;R1} be an arc of a link diagram $D$ admitting an integral region choice, 
and we denote by $D_+$ and $D_-$ the obtained diagrams from $D$ by a Reidemeister move I at the arc 
as illustrated on the right side and the left side of Figure \ref{Fig;R1} respectively. 
We may order the regions of $D$ 
such that the upper and lower regions on the middle of Figure \ref{Fig;R1}  are ordered 1 and 2 respectively. 
We denote the definite and the alternating region choice matrices of the double counting rule for $D$ 
by   
\[
A_{d2}(D)=
\begin{pmatrix}
\mathbf{a}_d & \mathbf{b}_d & P_d 
\end{pmatrix}
, \ \ \ 
A_{a2}(D)=
\begin{pmatrix}
\mathbf{a}_a & \mathbf{b}_a & P_a 
\end{pmatrix}
. 
\]
Inserting the row and the column corresponding to the added crossing and the added region respectively, 
we obtain the matrices for $D_+$ and $D_-$,   
\[
A_{d2}(D_+)=A_{d2}(D_-)=
\begin{pmatrix}
1 & 2 & 1 & \mathbf{0} \\
\mathbf{0} & \mathbf{a}_d & \mathbf{b}_d & P_d 
\end{pmatrix}
, 
\]
\[
A_{a2}(D_+)=
\begin{pmatrix}
1 & -2 & 1 & \mathbf{0} \\
\mathbf{0} & \mathbf{a}_a & \mathbf{b}_a & P_a 
\end{pmatrix}
, \ \ \ 
A_{a2}(D_-)=
\begin{pmatrix}
-1 & 2 & -1 & \mathbf{0} \\
\mathbf{0} & \mathbf{a}_a & \mathbf{b}_a & P_a 
\end{pmatrix}
. 
\]
Each of the matrices for $D_+$ and $D_-$ has one more column than $D$, 
and 
the equalities $\mathrm{rank}A_{d2}(D_{\pm})=\mathrm{rank}A_{d2}(D)+1$ 
and $\mathrm{rank}A_{a2}(D_{\pm})=\mathrm{rank}A_{a2}(D)+1$ hold. 
Then the $\mathbb{Z}$-submodules 
of the kernel solutions for these matrices have same rank. 
\end{proof}

\begin{lemma}
\label{lem;R2mat}
A Reidemeister move II between link diagrams  admitting integral region choices 
preserves the rank 
of the $\mathbb{Z}$-submodules 
of kernel solutions for 
definite and alternating region choice matrices of the double counting rule. 
\end{lemma}

\begin{figure}[htbp]
\begin{center}
\includegraphics[height=1.5cm,clip]{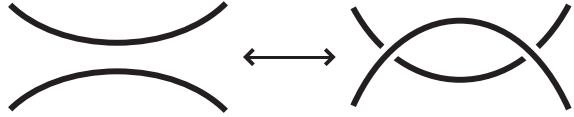}
\end{center}
\caption{A Reidemeister move II between the diagrams $D$ and $D'$.}
\label{Fig;R2}
\end{figure}

\begin{proof}
(cf. \cite{aharasuzuki})
Let the left side of Figure \ref{Fig;R2} be two arcs of a link diagram $D$ admitting an integral region choice, 
and we denote by $D'$ 
the obtained diagram from $D$ by a Reidemeister move II around the arcs 
as illustrated on the right side of Figure \ref{Fig;R2}. 

If the Reidemeister move II does not change the number of the connected components of the diagram, 
and the regions appearing around the move are different each other, 
then we may order the regions of $D$ 
such that the bottom, middle and top regions of the left side of Figure \ref{Fig;R2} are ordered 1, 2 and 3 respectively. 
We denote the definite and the alternating region choice matrices of the double counting rule for $D$ 
by   
\[
A_{d2}(D)=
\begin{pmatrix}
\mathbf{a}_d & \mathbf{b}_d & \mathbf{c}_d & P_d 
\end{pmatrix}
, \ \ \ 
A_{a2}(D)=
\begin{pmatrix}
\mathbf{a}_a & \mathbf{b}_a & \mathbf{c}_a & P_a 
\end{pmatrix}
.
\]
Inserting the rows and the column corresponding to the added crossings and the added region respectively, 
we obtain the matrices for $D'$,   
\[
A_{d2}(D')=
\begin{pmatrix}
1 & 1 & 1 & 0 & 1 & O \\
1 & 1 & 0 & 1 & 1 & \  \\
\mathbf{0} & \mathbf{a}_d & \mathbf{b}_d' & \mathbf{b}_d'' & \mathbf{c}_d & P_d 
\end{pmatrix}
\]
where $\mathbf{b}_d'+\mathbf{b}_d''=\mathbf{b}_d$, and  
\[
A_{a2}(D')=
\begin{pmatrix}
1 & -1 & 1 & 0 & -1 & O \\
-1 & 1 & 0 & -1 & 1 & \  \\
\mathbf{0} & \mathbf{a}_a & \mathbf{b}_a' & \mathbf{b}_a'' & \mathbf{c}_a & P_a 
\end{pmatrix}
\]
where $\mathbf{b}_a'+\mathbf{b}_a''=\mathbf{b}_a$. 
Each of the matrices for $D'$ has two more columns than $D$. 
Adding the fourth column to the third column, 
and taking the first column off the second, third and fifth columns on $A_{d2}(D')$, 
we obtain the matrix
\[
\begin{pmatrix}
1 & 0 & 0 & 0 & 0 & O \\
1 & 0 & 0 & 1 & 0 & \  \\
\mathbf{0} & \mathbf{a}_d & \mathbf{b}_d & \mathbf{b}_d'' & \mathbf{c}_d & P_d 
\end{pmatrix}
, 
\]
and the equality  $\mathrm{rank}A_{d2}(D')=\mathrm{rank}A_{d2}(D)+2$.  
Adding the fourth column to the third column, and the first column to the second and fifth columns, 
and taking the first column off the third column on $A_{a2}(D')$, 
we obtain the matrix
\[
\begin{pmatrix}
1 & 0 & 0 & 0 & 0 & O \\
-1 & 0 & 0 & -1 & 0 & \  \\
\mathbf{0} & \mathbf{a}_a & \mathbf{b}_a & \mathbf{b}_a'' & \mathbf{c}_a & P_a 
\end{pmatrix}
, 
\]
and the equality  
 $\mathrm{rank}A_{a2}(D')=\mathrm{rank}A_{a2}(D)+2$. 
Then the $\mathbb{Z}$-submodules 
of the kernel solutions for these matrices have same rank.  

If the Reidemeister move II does not change the number of the connected components of the diagram, 
and the top and bottom regions  coincide,  
then we may order the regions of $D$ 
such that the upper and middle regions of the left side of Figure \ref{Fig;R2} are ordered 1,and 2 respectively. 
We denote the definite and the alternating region choice matrices of the double counting rule for $D$ 
by   
\[
A_{d2}(D)=
\begin{pmatrix}
\mathbf{a}_d & \mathbf{b}_d & P_d 
\end{pmatrix}
, \ \ \ 
A_{a2}(D)=
\begin{pmatrix}
\mathbf{a}_a & \mathbf{b}_a & P_a 
\end{pmatrix}
. 
\]
Inserting the rows and the column corresponding to the added crossings and the added region respectively, 
we obtain the matrices for $D'$,   
\[
A_{d2}(D')=
\begin{pmatrix}
1 & 2 & 1 & 0 & O \\
1 & 2 & 0 & 1 & \  \\
\mathbf{0} & \mathbf{a}_d & \mathbf{b}_d' & \mathbf{b}_d'' & P_d 
\end{pmatrix}
\]
where $\mathbf{b}_d'+\mathbf{b}_d''=\mathbf{b}_d$, and  
\[
A_{a2}(D')=
\begin{pmatrix}
1 & -2 & 1 & 0 & O \\
-1 & 2 & 0 & -1 & \  \\
\mathbf{0} & \mathbf{a}_a & \mathbf{b}_a' & \mathbf{b}_a'' & P_a 
\end{pmatrix}
\]
where $\mathbf{b}_a'+\mathbf{b}_a''=\mathbf{b}_a$. 
Each of the matrices for $D'$ has two more columns than $D$. 
Adding the fourth column to the third column, 
taking the first column off the third column, 
and taking the first column multiplied by 2 off the second column on $A_{d2}(D')$, 
we obtain the matrix
\[
\begin{pmatrix}
1 & 0 & 0 & 0 & O \\
1 & 0 & 0 & 1 & \  \\
\mathbf{0} & \mathbf{a}_d & \mathbf{b}_d & \mathbf{b}_d'' & P_d 
\end{pmatrix}
, 
\]
and the equality  $\mathrm{rank}A_{d2}(D')=\mathrm{rank}A_{d2}(D)+2$.  
Adding the fourth column to the third column, 
the first column multiplied by 2 to the second column, 
and taking the first column off the third column on $A_{a2}(D')$, 
we obtain the matrix
\[
\begin{pmatrix}
1 & 0 & 0 & 0 & O \\
-1 & 0 & 0 & -1 & \  \\
\mathbf{0} & \mathbf{a}_a & \mathbf{b}_a & \mathbf{b}_a'' & P_a 
\end{pmatrix}
, 
\]
and the equality  
 $\mathrm{rank}A_{a2}(D')=\mathrm{rank}A_{a2}(D)+2$. 
Then the $\mathbb{Z}$-submodules 
of the kernel solutions for these matrices have same rank. 

If the Reidemeister move II changes the number of the connected components of the diagram, 
then we may order the regions of $D$ 
such that the bottom, middle and top regions of the left side of Figure \ref{Fig;R2} are ordered 1, 2 and 3 respectively. 
We denote the definite and the alternating region choice matrices of the double counting rule for $D$ 
by   
\[
A_{d2}(D)=
\begin{pmatrix}
\mathbf{a}_d & \mathbf{b}_d & \mathbf{c}_d & P_d 
\end{pmatrix}
, \ \ \ 
A_{a2}(D)=
\begin{pmatrix}
\mathbf{a}_a & \mathbf{b}_a & \mathbf{c}_a & P_a 
\end{pmatrix}
.
\]
Inserting the rows and the column corresponding to the added crossings and the added region respectively, 
we obtain the matrices for $D'$,   
\[
A_{d2}(D')=
\begin{pmatrix}
1 & 1 & 1 & 1 & O \\
1 & 1 & 1 & 1 & \  \\
\mathbf{0} & \mathbf{a}_d & \mathbf{b}_d & \mathbf{c}_d & P_d 
\end{pmatrix}
, \ \ \ 
A_{a2}(D')=
\begin{pmatrix}
1 & -1 & 1 & -1 & O \\
-1 & 1 & -1 & 1 & \  \\
\mathbf{0} & \mathbf{a}_a & \mathbf{b}_a & \mathbf{c}_a & P_a 
\end{pmatrix}
.
\]
Each of the matrices for $D'$ has one more column than $D$, 
and 
the equalities $\mathrm{rank}A_{d2}(D')=\mathrm{rank}A_{d2}(D)+1$ 
and $\mathrm{rank}A_{a2}(D')=\mathrm{rank}A_{a2}(D)+1$ hold. 
Then the $\mathbb{Z}$-submodules 
of the kernel solutions for these matrices have same rank. 

\end{proof}

\begin{lemma}
\label{lem;R3mat}
A Reidemeister moves III between link diagrams admitting integral region choices
preserves the ranks 
of the $\mathbb{Z}$-submodules 
of kernel solutions for 
definite and alternating region choice matrices of the double counting rule. 
\end{lemma}

\begin{figure}[htbp]
\begin{center}
\includegraphics[height=2.5cm,clip]{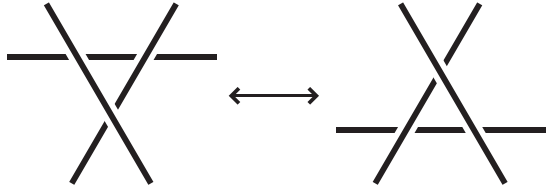}
\end{center}
\caption{A Reidemeister move III between the diagrams $D_{\nabla}$ and $D_{\Delta}$.}
\label{Fig;R3}
\end{figure}

\begin{proof}
(cf. \cite{aharasuzuki})
Let the left side of Figure \ref{Fig;R3} be 
a neighborhood of a triangle region on a link diagram $D_{\nabla}$ admitting an integral region choice, 
and we denote by $D_{\Delta}$ 
the obtained diagram from $D_{\nabla}$ by a Reidemeister move III around the triangle region  
as illustrated on the right side of Figure \ref{Fig;R3}.

If regions appearing around the move are different each other, 
then we may order the regions of $D_{\nabla}$ 
such that the triangle region on Figure \ref{Fig;R3} is ordered 1 
and that other six regions around are ordered 2, 3, 4, 5, 6, 7 clockwise from top left. 
The definite and the alternating region choice matrices of the double counting rule for $D_{\nabla}$ 
are    
\[
A_{d2}(D_{\nabla})=
\begin{pmatrix}
1 & 1 & 1 & 0 & 0 & 0 & 1 & \  \\
1 & 0 & 1 & 1 & 1 & 0 & 0 & O \\
1 & 0 & 0 & 0 & 1 & 1 & 1 & \  \\
\mathbf{0} & \mathbf{a}_d & \mathbf{b}_d & \mathbf{c}_d & \mathbf{d}_d & \mathbf{e}_d & \mathbf{f}_d & P_d 
\end{pmatrix}
, 
\]
\[
A_{a2}(D_{\nabla})=
\begin{pmatrix}
-1 & -1 & 1 & 0 & 0 & 0 & 1 & \  \\
1 & 0 & -1 & 1 & -1 & 0 & 0 & O \\
1 & 0 & 0 & 0 & -1 & 1 & -1 & \  \\
\mathbf{0} & \mathbf{a}_a & \mathbf{b}_a & \mathbf{c}_a & \mathbf{d}_a & \mathbf{e}_a & \mathbf{f}_a & P_a 
\end{pmatrix}
, 
\]
where the upper left crossing, the upper right crossing, and the lower crossing of the triangle of $D_{\nabla}$
are ordered 1,2, and 3. 
We order the crossings of $D_{\Delta}$ 
such that the lower right crossing, the lower left crossing, and the upper crossing are ordered 1,2, and 3, 
and that the others are ordered as $D_{\nabla}$.  
Then we obtain the matrices for $D_{\Delta}$, 
\[
A_{d2}(D_{\Delta})=
\begin{pmatrix}
1 & 0 & 0 & 1 & 1 & 1 & 0 & \  \\
1 & 1 & 0 & 0 & 0 & 1 & 1 & O \\
1 & 1 & 1 & 1 & 0 & 0 & 0 & \  \\
\mathbf{0} & \mathbf{a}_d & \mathbf{b}_d & \mathbf{c}_d & \mathbf{d}_d & \mathbf{e}_d & \mathbf{f}_d & P_d 
\end{pmatrix}
, 
\]
\[
A_{a2}(D_{\Delta})=
\begin{pmatrix}
-1 & 0 & 0 & 1 & -1 & 1 & 0 & \  \\
1 & -1 & 0 & 0 & 0 & -1 & 1 & O \\
1 & -1 & 1 & -1 & 0 & 0 & 0 & \  \\
\mathbf{0} & \mathbf{a}_a & \mathbf{b}_a & \mathbf{c}_a & \mathbf{d}_a & \mathbf{e}_a & \mathbf{f}_a & P_a 
\end{pmatrix}
. 
\]
After multiplying the top three rows on $A_{d_2}(D_{\nabla})$ by $-1$ and multiplying the first column by $-1$, 
if we add the first column to the $j$-th collumns, $j=2,3, \cdots , 7$, then we obtain $A_{d_2}(D_{\Delta})$.  
The matrices $A_{d2}(D_{\nabla})$ and $A_{d2}(D_{\Delta})$ have 
same size and same rank,  
then the $\mathbb{Z}$-submodules 
of the kernel solutions for the matrices have same rank. 
Adding the first column multiplied by $(-1)^{j+1}$ to the $j$-th columns, $j=2,3, \cdots , 7$, 
we obtain $A_{a2}(D_{\Delta})$ from $A_{a2}(D_{\nabla})$.  
The matrices $A_{a2}(D_{\nabla})$ and $A_{a2}(D_{\Delta})$ have 
same size and same rank,  
then the $\mathbb{Z}$-submodules 
of the kernel solutions for these matrices have same rank. 

If the top left and top right regions of the left side of Figure \ref{Fig;R3} coincide on $D_{\nabla}$, 
and if this and the other regions appearing around the move are different each other,    
then we may order the regions of $D_{\nabla}$ 
such that the triangle region on Figure \ref{Fig;R3} is ordered 1 
and that other five regions around are ordered 2, 3, 4, 5, 6 clockwise from top left. 
The definite and the alternating region choice matrices of the double counting rule for $D_{\nabla}$ 
are   
\[
A_{d2}(D_{\nabla})=
\begin{pmatrix}
1 & 1 & 1 & 0 & 0 & 1 & \  \\
1 & 1 & 1 & 1 & 0 & 0 & O \\
1 & 0 & 0 & 1 & 1 & 1 & \  \\
\mathbf{0} & \mathbf{a}_d & \mathbf{b}_d & \mathbf{d}_d & \mathbf{e}_d & \mathbf{f}_d & P_d 
\end{pmatrix}
, 
\]
\[
A_{a2}(D_{\nabla})=
\begin{pmatrix}
-1 & -1 & 1 & 0 & 0 & 1 & \  \\
1 & 1 & -1 & -1 & 0 & 0 & O \\
1 & 0 & 0 & -1 & 1 & -1 & \  \\
\mathbf{0} & \mathbf{a}_a & \mathbf{b}_a & \mathbf{d}_a & \mathbf{e}_a & \mathbf{f}_a & P_a 
\end{pmatrix}
. 
\]
Then we obtain the matrices for $D_{\Delta}$, 
\[
A_{d2}(D_{\Delta})=
\begin{pmatrix}
1 & 1 & 0 & 1 & 1 & 0 & \  \\
1 & 1 & 0 & 0 & 1 & 1 & O \\
1 & 2 & 1 & 0 & 0 & 0 & \  \\
\mathbf{0} & \mathbf{a}_d & \mathbf{b}_d & \mathbf{d}_d & \mathbf{e}_d & \mathbf{f}_d & P_d 
\end{pmatrix}
,
\]
\[
A_{a2}(D_{\Delta})=
\begin{pmatrix}
-1 & 1 & 0 & -1 & 1 & 0 & \  \\
1 & -1 & 0 & 0 & -1 & 1 & O \\
1 & -2 & 1 & 0 & 0 & 0 & \  \\
\mathbf{0} & \mathbf{a}_a & \mathbf{b}_a & \mathbf{d}_a & \mathbf{e}_a & \mathbf{f}_a & P_a 
\end{pmatrix}
.
\]
After multiplying the top three rows on $A_{d_2}(D_{\nabla})$ by $-1$ and multiplying the first column by $-1$, 
if we add the first column to the $j$-th collumns, $j=3, \cdots , 6$ 
and the first column multiplied by 2 to the second column, 
then we obtain $A_{d_2}(D_{\Delta})$.  
The matrices $A_{d2}(D_{\nabla})$ and $A_{d2}(D_{\Delta})$ have 
same size and same rank,  
then the $\mathbb{Z}$-submodules 
of the kernel solutions for the matrices have same rank. 
Adding the first column to the third, fourth, and sixth columns,
and taking the first column off the fifth column 
and the first column multiplied by 2 off the second column, 
we obtain $A_{a2}(D_{\Delta})$ from $A_{a2}(D_{\nabla})$.  
The matrices $A_{a2}(D_{\nabla})$ and $A_{a2}(D_{\Delta})$ have 
same size and same rank,  
then the $\mathbb{Z}$-submodules 
of the kernel solutions for these matrices have same rank. 

If the top left, top right, and the bottom middle regions of the left side of Figure \ref{Fig;R3} coincide on $D_{\nabla}$, 
then  this and the other regions appearing around the move are different each other. 
We may order the regions of $D_{\nabla}$ 
such that the triangle region on Figure \ref{Fig;R3} is ordered 1 
and that other four regions around are ordered 2, 3, 4, 5 clockwise from top left. 
The definite and the alternating region choice matrices of the double counting rule for $D_{\nabla}$ 
are   
\[
A_{d2}(D_{\nabla})=
\begin{pmatrix}
1 & 1 & 1 & 0 & 1 & \  \\
1 & 1 & 1 & 1 & 0 & O \\
1 & 1 & 0 & 1 & 1 & \  \\
\mathbf{0} & \mathbf{a}_d & \mathbf{b}_d & \mathbf{d}_d & \mathbf{f}_d & P_d 
\end{pmatrix}
, 
\]
\[
A_{a2}(D_{\nabla})=
\begin{pmatrix}
-1 & -1 & 1 & 0 & 1 & \  \\
1 & 1 & -1 & -1 & 0 & O \\
1 & 1 & 0 & -1 & -1 & \  \\
\mathbf{0} & \mathbf{a}_a & \mathbf{b}_a & \mathbf{d}_a & \mathbf{f}_a & P_a 
\end{pmatrix}
. 
\]
Then we obtain the matrices for $D_{\Delta}$, 
\[
A_{d2}(D_{\Delta})=
\begin{pmatrix}
1 & 2 & 0 & 1 & 0 & \  \\
1 & 2 & 0 & 0 & 1 & O \\
1 & 2 & 1 & 0  & 0 & \  \\
\mathbf{0} & \mathbf{a}_d & \mathbf{b}_d & \mathbf{d}_d & \mathbf{f}_d & P_d 
\end{pmatrix}
,
\]
\[
A_{a2}(D_{\Delta})=
\begin{pmatrix}
-1 & 2 & 0 & -1 & 0 & \  \\
1 & -2 & 0 & 0  & 1 & O \\
1 & -2 & 1 & 0 & 0 & \  \\
\mathbf{0} & \mathbf{a}_a & \mathbf{b}_a & \mathbf{d}_a & \mathbf{f}_a & P_a 
\end{pmatrix}
.
\]
After multiplying the top three rows on $A_{d_2}(D_{\nabla})$ by $-1$ and multiplying the first column by $-1$, 
if we add the first column to the third, fourth and fifth collumns, 
and the first column multiplied by 3 to the second column, 
then we obtain $A_{d_2}(D_{\Delta})$.  
The matrices $A_{d2}(D_{\nabla})$ and $A_{d2}(D_{\Delta})$ have 
same size and same rank,  
then the $\mathbb{Z}$-submodules 
of the kernel solutions for the matrices have same rank. 
Adding the first column to the third, fourth, and fifth columns,
and taking the first column multiplied by 3 off the second columns, 
we obtain $A_{a2}(D_{\Delta})$ from $A_{a2}(D_{\nabla})$.  
The matrices $A_{a2}(D_{\nabla})$ and $A_{a2}(D_{\Delta})$ have 
same size and same rank,  
then the $\mathbb{Z}$-submodules 
of the kernel solutions for these matrices have same rank. 

If the top left and top right regions of the left side of Figure \ref{Fig;R3} coincide on $D_{\nabla}$, 
and if the bottom left and bottom right regions also coincide on $D_{\nabla}$, 
we may  order the regions of $D_{\nabla}$ 
such that the triangle region on Figure \ref{Fig;R3} is ordered 1 
and that other four regions around are ordered 2, 3, 4, 5 clockwise from top left. 
The definite and the alternating region choice matrices of the double counting rule for $D_{\nabla}$ 
are   
\[
A_{d2}(D_{\nabla})=
\begin{pmatrix}
1 & 1 & 1 & 1 & 0 & \  \\
1 & 1 & 1 & 1 & 0 & O \\
1 & 0 & 0 & 2 & 1 & \  \\
\mathbf{0} & \mathbf{a}_d & \mathbf{b}_d & \mathbf{d}_d & \mathbf{e}_d & P_d 
\end{pmatrix}
, 
\]
\[
A_{a2}(D_{\nabla})=
\begin{pmatrix}
-1 & -1 & 1 & 1 & 0 & \  \\
1 & 1 & -1 & -1 & 0 & O \\
1 & 0 & 0 & -2 & 1 & \  \\
\mathbf{0} & \mathbf{a}_a & \mathbf{b}_a & \mathbf{d}_a & \mathbf{e}_a  & P_a 
\end{pmatrix}
. 
\]
Then we obtain the matrices for $D_{\Delta}$, 
\[
A_{d2}(D_{\Delta})=
\begin{pmatrix}
1 & 1 & 0 & 1 & 1 & \  \\
1 & 1 & 0 & 1 & 1 & O \\
1 & 2 & 1 & 0 & 0 & \  \\
\mathbf{0} & \mathbf{a}_d & \mathbf{b}_d & \mathbf{d}_d & \mathbf{e}_d & P_d 
\end{pmatrix}
,
\]
\[
A_{a2}(D_{\Delta})=
\begin{pmatrix}
-1 & 1 & 0 & -1 & 1 & \  \\
1 & -1 & 0 & 1 & -1 & O \\
1 & -2 & 1 & 0 & 0 & \  \\
\mathbf{0} & \mathbf{a}_a & \mathbf{b}_a & \mathbf{d}_a & \mathbf{e}_a & P_a 
\end{pmatrix}
.
\]
After multiplying the top three rows on $A_{d_2}(D_{\nabla})$ by $-1$ and multiplying the first column by $-1$, 
if we add the first column to the third and fifth collumns, 
and the first column multiplied by 2 to the second and fourth columns, 
then we obtain $A_{d_2}(D_{\Delta})$.  
The matrices $A_{d2}(D_{\nabla})$ and $A_{d2}(D_{\Delta})$ have 
same size and same rank,  
then the $\mathbb{Z}$-submodules 
of the kernel solutions for the matrices have same rank. 
Adding the first column multiplied by $-2$, $1$, $2$ and $-1$ to 
the second, third, fourth, and fifth columns respectively, 
we obtain $A_{a2}(D_{\Delta})$ from $A_{a2}(D_{\nabla})$.  
The matrices $A_{a2}(D_{\nabla})$ and $A_{a2}(D_{\Delta})$ have 
same size and same rank,  
then the $\mathbb{Z}$-submodules 
of the kernel solutions for these matrices have same rank. 

The proof for the other cases are given by the similar arguments as above, 
or reduced to the one of the above cases applying Lemma \ref{lem;ccAZRCmat}.  
\end{proof}

\begin{proof}[Proof of Theorem \ref{thm;rank}.]
It is well known that  
we can make any link diagram into a diagram of a trivial link after some crossing changes, 
and that 
some Reidemeister moves can transform the non-trivial diagram of a trivial link 
to the split sum of the knot diagram with only one crossing and the $l-1$ copies of the trivial knot daigaram, 
which is given in Example \ref{ex;n1}. 
On this obtained diagram $D_0$ with certain orders of crossings and regions, 
we obtain 
\[
A_{d2}(D_0)=
\begin{pmatrix}
2        & 1 & 1 & 0 & 0 & 0 & \ldots & 0 
\end{pmatrix}
,
\]
and 
\[
A_{a2}(D_0)=
\begin{pmatrix}
-2\varepsilon & \varepsilon & \varepsilon & 0 & 0 & 0 & \ldots & 0  
\end{pmatrix}
,
\]
where $\varepsilon =1$ if the crossing is positive, otherwise $\varepsilon =-1$, 
and the number of $0$ appearing on each matrix is $l-1$. 
The both matrices have $l+2$ columns. 
Their ranks are equal to $1=1+l-l$. 
Then the desired claim holds for $D_0$. 
By Lemma \ref{lem;ccAZRCmat}, \ref{lem;R1mat}, \ref{lem;R2mat}, and \ref{lem;R3mat}, 
crossing changes and Reidemeister moves preserve this claim. 
Then this theorem is proved.   
\end{proof}

\begin{remark}
As commented in Remark \ref{rem;arcfixAZRCker} and \ref{rem;arcfixDZRCker}, 
Ahara and Suzuki \cite{aharasuzuki} and Harada \cite{haradaM}  
showed that Reidemeister moves and crossing changes preserve 
the existence of the kernel solution, and that 
the knot diagram with only one crossing has a kernel solution, 
to prove Lemma \ref{lem;arcfixAZRCker}  and \ref{lem;arcfixDZRCker} for a knot diagram. 
Their arguments are extended to that for the proof of  Theorem \ref{thm;rank}, 
though Harada did not describe how alternating region choice matrices are affected. 
\end{remark}

\section{Region choice matrices of the single counting rule}\label{sect;RCmatrixS}


Before determining the rank of region choice matrices of the single counting rule, 
we observe the images of 
the homomorphisms $\Phi _{d1}(D), \Phi _{d2}(D), \Phi _{a1}(D), \Phi _{a2}(D)$ 
defined in Section \ref{sect;kernel}.
The arguments in the proofs of Theorem \ref{thm;aharasuzuki;mat} (1) and \ref{thm;harada;mat} (1) 
given in Section \ref{sect;DZRCPpf} and \ref{sect;AZRCPpf} 
imply the following result. 

\begin{lemma}\label{lem;Im2to1}
Let $D$ be a diagram of an $l$-component link. 
We assume that $D$ has $d$ connected components and $n$ crossings, 
$n\geq 1$.  
We take $\mathbf{c}\in \mathbb{Z}^n$. 
\begin{enumerate}
\item 
If there exists a solution $\mathbf{w}\in \mathbb{Z}^{n+d+1}$ 
such that $A_{d2}(D)\mathbf{w}+\mathbf{c}=\mathbf{0}$,  
then 
there exists a solution $\mathbf{u}\in \mathbb{Z}^{n+d+1}$ 
such that $A_{d1}(D)\mathbf{u}+\mathbf{c}=\mathbf{0}$.  
\item 
If there exists a solution $\mathbf{w}\in \mathbb{Z}^{n+d+1}$ 
such that $A_{a2}(D)\mathbf{w}+\mathbf{c}=\mathbf{0}$,  
then 
there exists a solution $\mathbf{u}\in \mathbb{Z}^{n+d+1}$ 
such that $A_{a1}(D)\mathbf{u}+\mathbf{c}=\mathbf{0}$.  
\end{enumerate}
\hfill $\square$
\end{lemma}

By a similar argument, 
the converse of Lemma \ref{lem;Im2to1} is shown as follows. 

\begin{lemma}\label{lem;Im1to2}
Let $D$ be a diagram of an $l$-component link. 
We assume that $D$ has $d$ connected components and $n$ crossings $x_1, \cdots , x_n$, 
$n\geq 1$.  
Let $R_1, \cdots , R_{n+d+1}$ be the regions of $D$. 
We take $\mathbf{c}\in \mathbb{Z}^n$. 
\begin{enumerate}
\item 
If there exists a solution $\mathbf{u}\in \mathbb{Z}^{n+d+1}$ 
such that $A_{d1}(D)\mathbf{u}+\mathbf{c}=\mathbf{0}$,  
then 
there exists a solution $\mathbf{w}\in \mathbb{Z}^{n+d+1}$ 
such that $A_{d2}(D)\mathbf{w}+\mathbf{c}=\mathbf{0}$.  
\item 
If there exists a solution $\mathbf{u}\in \mathbb{Z}^{n+d+1}$ 
such that $A_{a1}(D)\mathbf{u}+\mathbf{c}=\mathbf{0}$,  
then 
there exists a solution $\mathbf{w}\in \mathbb{Z}^{n+d+1}$ 
such that $A_{a2}(D)\mathbf{w}+\mathbf{c}=\mathbf{0}$.  
\end{enumerate}
\end{lemma}

\begin{proof}
If the given diagram $D$ is irreducible, 
then we have $A_{d1}(D)=A_{d2}(D)$ and $A_{a1}(D)=A_{a2}(D)$. 

We suppose that the diagram $D$ has at least one reducible crossing. 
For a region $R_j$, 
let $\mathcal{X}_j$ be the set of reducible crossings touched by $R_j$ twice, $j=1, \cdots , n+d+1$. 
A set $\mathcal{X}_j$ might be empty. 
We take $\mathbf{r}_j\in \mathbb{Z}^{n+d+1}$ such that 
any components of $\mathbf{r}_j$ are $0$ 
but the $j$-th component is $1$. 
Choosing $R_j$ once corresponds with $\mathbf{r}_j$. 
\begin{enumerate}
\item 
We suppose that $A_{d1}(D)\mathbf{u}+\mathbf{c}=\mathbf{0}$. 
Let $u_j$ be the $j$-th component of $\mathbf{u}$, $j=1, \cdots , n+d+1$. 
Applying Theorem \ref{thm;add1DZRC} 
for each reducible crossing $y\in \mathcal{X}_j$, 
we obtain $\mathbf{v}'_y \in \mathbb{Z}^{n+d+1}$ such that 
any components of $A_{d2}(D)\mathbf{v}'_y$ are $0$ 
but the component of $A_{d2}(D)\mathbf{v}'_y$ to $y$ is $1$. 
By the definitions of the definite region choice matrices, 
we have 
\[
 A_{d2}(D)\mathbf{r}_j - A_{d1}(D)\mathbf{r}_j =\sum _{y\in \mathcal{X}_j} A_{d2}(D)\mathbf{v}'_y.
\] 
We note $\displaystyle \mathbf{u}=\sum _j u_j\mathbf{r}_j$. 
We take $\displaystyle \mathbf{w}=\mathbf{u}-\sum _j u_j \sum _{y\in \mathcal{X}_j} \mathbf{v}'_y$. 
Then we have 
\begin{eqnarray*}
A_{d2}(D)\mathbf{w}
&=& A_{d2}(D)\mathbf{u}-\sum _j u_j \sum _{y\in \mathcal{X}_j} A_{d2}(D)\mathbf{v}'_y \\
&=& \sum _j u_j \left( A_{d2}(D)\mathbf{r}_j - \sum _{y\in \mathcal{X}_j} A_{d2}(D)\mathbf{v}'_y \right) \\
&=& \sum _j u_j A_{d1}(D)\mathbf{r}_j \\
&=& A_{d1}(D)\mathbf{u} \\
&=& -\mathbf{c}. 
\end{eqnarray*}
Therefore we have $A_{d2}(D)\mathbf{w} +\mathbf{c}=\mathbf{0}$. 

\item 
We suppose that $A_{a1}(D)\mathbf{u}+\mathbf{c}=\mathbf{0}$. 
Let $u_j$ be the $j$-th component of $\mathbf{u}$, $j=1, \cdots , n+d+1$. 
Applying Theorem \ref{thm;add1AZRC} 
for each reducible crossing $y\in \mathcal{X}_j$, 
we obtain $\mathbf{v}'_y \in \mathbb{Z}^{n+d+1}$ such that 
any components of $A_{a2}(D)\mathbf{v}'_y$ are $0$ 
but the component of $A_{a2}(D)\mathbf{v}'_y$ to $y$ is $1$. 
By the definitions of the definite region choice matrices, 
we have 
\[
 A_{a2}(D)\mathbf{r}_j - A_{a1}(D)\mathbf{r}_j =\sum _{y\in \mathcal{X}_j} A_{a2}(D)\mathbf{v}'_y.
\] 
We note $\displaystyle \mathbf{u}=\sum _j u_j\mathbf{r}_j$. 
We take $\displaystyle \mathbf{w}=\mathbf{u}-\sum _j u_j \sum _{y\in \mathcal{X}_j} \mathbf{v}'_y$. 
Then we have 
\begin{eqnarray*}
A_{a2}(D)\mathbf{w}
&=& A_{a2}(D)\mathbf{u}-\sum _j u_j \sum _{y\in \mathcal{X}_j} A_{a2}(D)\mathbf{v}'_y \\
&=& \sum _j u_j \left( A_{a2}(D)\mathbf{r}_j - \sum _{y\in \mathcal{X}_j} A_{a2}(D)\mathbf{v}'_y \right) \\
&=& \sum _j u_j A_{a1}(D)\mathbf{r}_j \\
&=& A_{a1}(D)\mathbf{u} \\
&=& -\mathbf{c}. 
\end{eqnarray*}
Therefore we have $A_{a2}(D)\mathbf{w} +\mathbf{c}=\mathbf{0}$. 
\end{enumerate}
\end{proof}

By Lemma \ref{lem;Im2to1} and \ref{lem;Im1to2}, 
we obtain the following result. 

\begin{theorem}\label{thm;image}
Let $D$ be a link diagram. 
We assume that $D$ has at least one crossing. 
\begin{enumerate}
\item 
The image of the homomorphism $\Phi _{d1}(D)$ coincides with 
the image of the homomorphism $\Phi _{d2}(D)$. 
\item 
The image of the homomorphism $\Phi _{a1}(D)$ coincides with 
the image of the homomorphism $\Phi _{a2}(D)$. 
\end{enumerate}
\hfill $\square$
\end{theorem}

The above theorem implies that 
the existence of a solution of the single counting rule 
coincides with 
that of the double counting rule  
for each of the both integral region choice problem. 
Particularly, the first and second results of Theorem \ref{thm;aharasuzuki;mat} are equivalent, 
and the first and second results of Theorem \ref{thm;harada;mat} are equivalent. 

By Theorem \ref{thm;rank}, \ref{thm;image}, and the homomorphism theorem, 
we obtain the following theorem. 

\begin{theorem}\label{thm;rankS}
Let $D$ be a diagram of an $l$-component link. 
We assume that $D$ has $d$ connected components and $n$ crossings, 
$n\geq 1$.  
Then each rank of the definite and alternating region choice matrices of the single counting rule, 
$A_{d1}(D)$ and $A_{a1}(D)$, 
is $n+d-l$,  
and each 
rank of the $\mathbb{Z}$-submodules 
$\{ \mathbf{u} \in \mathbb{Z}^{n+d+1} \mid A_{d1}(D) \mathbf{u}=\mathbf{0} \}$
and 
$\{ \mathbf{u} \in \mathbb{Z}^{n+d+1} \mid A_{a1}(D) \mathbf{u}=\mathbf{0} \}$
is $l+1$.
\hfill $\square$
\end{theorem}

In \cite{chenggao}, Chen and Gao determined 
the $\mathbb{Z}_2$-rank of the incidence matrix for connected link diagrams. 
In \cite{hashizume2013}, Hashizume generalized their result to 
disconnected diagrams and connected diagrams,  and determined 
the rank, which she called the $\mathbb{Z}_2$-dimension, 
of the $\mathbb{Z}_2$-submodule  
of kernel solutions 
for the homomorphism induced from region crossing changes  
on 
diagrams. 
The ranks obtained in Theorem \ref{thm;rank} and \ref{thm;rankS} 
are same values as their ranks. 
If we transpose their incidence matrix, 
it coincides with the definite region choice matrix of the single counting rule 
up to permutations of rows and columns. 
This transposed matrix also coincides with the alternating region choice matrix of the single counting rule 
modulo $2$ up to permutations of rows and columns. 
Hence Theorem \ref{thm;rank} and \ref{thm;rankS} are integral extensions of their results. 

\section{Standard kernel solutions of the double counting rule}\label{sect;basis}


In \cite{hashizume2013, hashizume2015}, 
Hashizume studied structures of 
the $\mathbb{Z}_2$-homomorphism 
induced by region crossing changes on link diagrams. 
Particularly, 
she gave a basis of the kernel of the homomorphism 
on an irreducible link diagram. 
In this section, 
we observe the kernels of the $\mathbb{Z}$-homomorphisms 
$\Phi _{a2}$ and $\Phi _{d2}$ given in Section \ref{sect;kernel}.

\begin{lemma}\label{lem;constKer}
On any link diagram with at least one crossing, 
assigning a same integer to all regions 
gives a kernel solution for 
an alternating region choice matrix of the double counting rule. 
\end{lemma}

\begin{proof}
Any crossing is touched by four corners of regions, 
and we have $p-p+p-p=0$ for any integer $p$. 
\end{proof}

We denote by $\mathbf{u}_{\infty}$ the kernel solution assigning $1$ to all regions, 
as illustrated in Figure \ref{Fig;u_infty}. 

\begin{figure}[htbp]
\begin{center}
\includegraphics[height=3.5cm,clip]{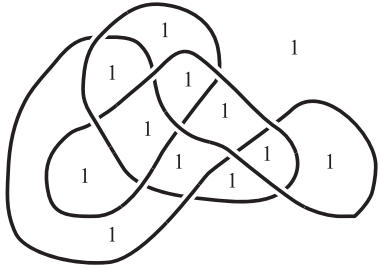}
\end{center}
\caption{The kernel solution $\mathbf{u}_{\infty}$.}
\label{Fig;u_infty}
\end{figure}

Let $D$ be an oriented link diagram with ordered link components, 
and $D_i$ be a sub-diagram of $D$ representing $i$-th link component, $i=1, \cdots , l$. 
If the oriented link diagram $D$ is a diagram on the plane $\mathbb{R}^2$, 
we may denote the unbounded region by $R_{\infty}(D)$.  
If $D$ is a diagram on the sphere $S^2=\mathbb{R}^2\cup \{ \infty \}$, 
we may denote the region including the infinite point by $R_{\infty}(D)$. 
For each sub-diagram $D_i$, we 
denote by $\mathbf{u}_i$ 
the kernel solution obtained by Lemma \ref{lem;compAlex} from 
a componentwise Alexander numbering associated with $D_i$ 
such that the region $R_{\infty}(D)$  is assigned $0$.   
We shall call $\mathbf{u}_i$ the \emph{standard kernel solution associated with} $D_i$.
Figure \ref{Fig;standardker}, same as Figure \ref{Fig;compAlex}, gives 
examples of standard kernel solutions 
on the same link diagram as Figure \ref{Fig;u_infty} and \ref{Fig;arcfixAZRCker}. 
The kernel solution given by an Alexander numbering for $D$ is equal to 
$\displaystyle r\mathbf{u}_{\infty}+\sum _{i=1}^l \mathbf{u}_i$ 
where the region $R_{\infty}(D)$ is assigned the integer $r$ in the Alexander numbering.

\begin{figure}[htbp]
\begin{center}
\includegraphics[height=3.5cm,clip]{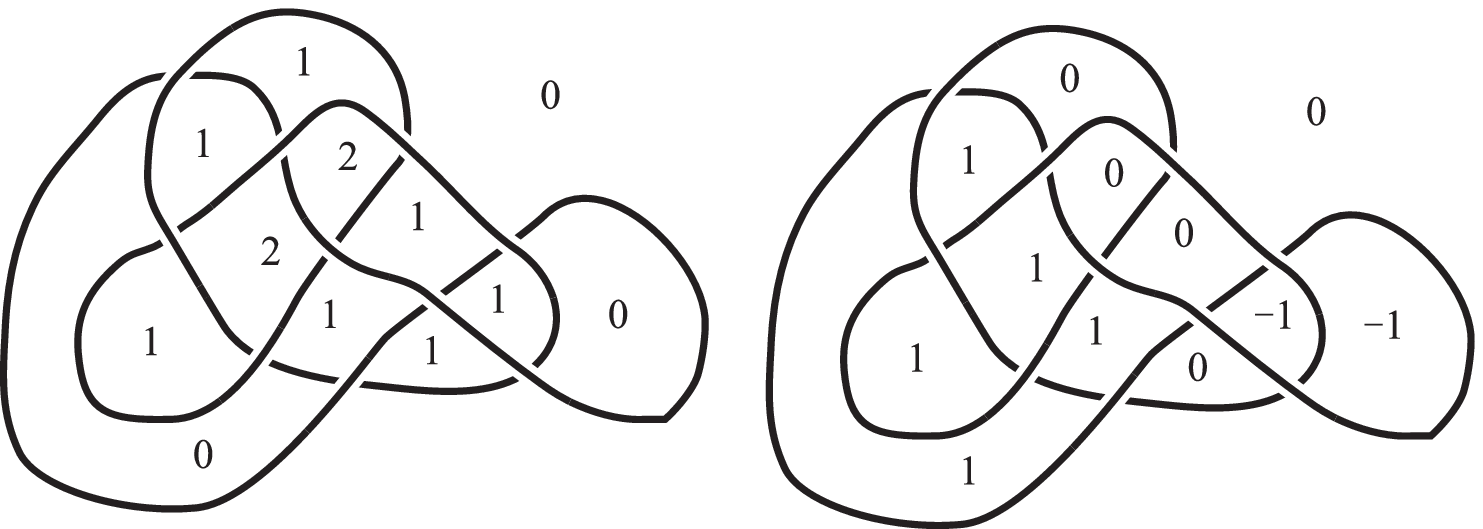}
\end{center}
\caption{The standard kernel solutions $\mathbf{u}_1$ and $\mathbf{u}_2$.}
\label{Fig;standardker}
\end{figure}

\begin{theorem}\label{thm;KerBasis}
Let $D$ be an oriented link diagram with $l$ ordered link components and at least one crossing, 
and $R_{\infty}=R_{\infty}(D)$ be the above region.  
The set of the 
above kernel solutions 
$\mathbf{u}_1, \cdots , \mathbf{u}_l$,  and $\mathbf{u}_{\infty}$ 
is a basis of the kernel of the homomorphism 
induced by the alternating integral region choice problem of double counting rule. 
\end{theorem}

\begin{proof}
For the linear independence of $\mathbf{u}_1, \cdots , \mathbf{u}_l, \mathbf{u}_{\infty}$, 
it is sufficient to prove 
that the standard kernel solutions are linearly independent, 
since the region $R_{\infty}$ is assigned $1$ by $\mathbf{u}_{\infty}$ 
and $0$ by $\mathbf{u}_i$, $i=1, \cdots , l$. 
We use an induction on $l$. 
If $D$ is a knot diagram, the standard kernel solution associated with $D$ 
has at least one component equal to $1$ or $-1$. 
Then it is linearly independent. 
We assume that $l\geq 2$ and that 
the standard kernel solutions 
are linearly independent on any oriented link diagram with less components than $l$. 
Let $D$ be an oriented link diagram with $l$ components.
We take a point $p$ on $\partial R_{\infty}(D)$ except crossings. 
We may assume that $p$ lies on $l$-th link component diagram $D_l$. 
We suppose $\displaystyle \sum _{i=1}^l n_i \mathbf{u}_i =\mathbf{0}$, $n_i \in \mathbb{Z}$.  
Let $R_p$ be a region of $D$ with $p\in \partial R_p$ and $R_p\neq R_{\infty}$. 
The region $R_p$ is assigned $1$ or $-1$ by $\mathbf{u}_l$ 
and $0$ by each of the other standard kernel solutions. 
Then we have $n_l=0$ and  $\displaystyle \sum _{i=1}^{l-1} n_i \mathbf{u}_i =\mathbf{0}$. 
On the diagram obtained from $D$ ignoring $D_l$, 
the standard kernel solutions are linearly independent by the assumption of the induction.  
Then $n_1, \cdots , n_{l-1}$ should be $0$. 
Hence the standard kernel solutions on $D$ are linearly independent. 

Let  $\mathbf{x}$ be a kernel solution of  the homomorphism $\Phi _{a2}(D)$. 
By Theorem \ref{thm;rank}, 
the rank 
of the kernel of the homomorphism $\Phi _{a2}(D)$ is $l+1$. 
Then $\mathbf{x}, \mathbf{u}_1, \cdots , \mathbf{u}_l, \mathbf{u}_{\infty}$ are linearly dependent, 
since $\mathbf{u}_1, \cdots , \mathbf{u}_l, \mathbf{u}_{\infty}$ are linearly independent. 
Hence 
$\displaystyle \mathbf{x}=\sum _{i=1}^l q_i\mathbf{u}_i+q_{\infty}\mathbf{u}_{\infty}$ holds  
for certain rational numbers $q_1, \cdots , q_l, q_{\infty}$. 
We show that $q_1, \cdots , q_l, q_{\infty}$ are integers.  
The region $R_{\infty}$ is assigned $1$ by $\mathbf{u}_{\infty}$ 
and $0$ by $\mathbf{u}_i$, $i=1, \cdots , l$.  
Hence 
$q_{\infty}$ is an integer since the components of $\mathbf{x}$ are integers. 
Then it is sufficient to show that $q_1, \cdots , q_l$ are integers 
if the all components of $\displaystyle \sum _{i=1}^l q_i \mathbf{u}_i$ are integers. 
We use an induction on $l$. 
If $D$ is a knot diagram, the standard kernel solution associated with $D$ 
has at least one component equal to $1$ or $-1$. 
Then $q_1\in \mathbb{Z}$ holds. 
We assume that $l\geq 2$ and that 
the desired claim holds for any oriented link diagram with less components than $l$. 
Let $D$ be an oriented link diagram with $l$ components.
The above region $R_p$ is assigned $1$ or $-1$ by $\mathbf{u}_l$ 
and $0$ by each of the other standard kernel solutions. 
Hence $q_l\in \mathbb{Z}$ holds  
and all components of $\displaystyle \sum _{i=1}^{l-1} q_i \mathbf{u}_i$ are integers. 
We apply the assumption of the induction to the diagram obtained from $D$ ignoring $D_l$.  
Then we have $q_1, \cdots , q_{l-1}\in \mathbb{Z}$. 

Therefore 
$\mathbf{x}$ is a linear combination of $\mathbf{u}_1, \cdots , \mathbf{u}_l, \mathbf{u}_{\infty}$ over $\mathbb{Z}$. 
Then the set of the kernel solutions 
$\mathbf{u}_1, \cdots , \mathbf{u}_l$,  and $\mathbf{u}_{\infty}$ 
is a basis of the kernel of the homomorphism $\Phi _{a2}(D)$. 
\end{proof}

For the above link diagram $D$, 
we fix a checkerboard coloring. 
We apply Lemma \ref{lem;KerAtoD} to the above basis 
$\mathbf{u}_1, \cdots , \mathbf{u}_l, \mathbf{u}_{\infty}$.  
Then we obtain 
kernel solutions 
of the definite integral region choice problem of double counting rule. 
We denote them by 
$\bar{\mathbf{u}}_1, \cdots , \bar{\mathbf{u}}_l, \bar{\mathbf{u}}_{\infty}$. 

\begin{theorem}\label{thm;KerBasisD}
Let $D$ be an oriented link diagram with $l$ ordered link components and at least one crossing, 
and $R_{\infty}=R_{\infty}(D)$ be the above region.  
The set of the above kernel solutions $\bar{\mathbf{u}}_1, \cdots , \bar{\mathbf{u}}_l$,  and $\bar{\mathbf{u}}_{\infty}$ 
is a basis of the kernel of the homomorphism 
induced by the definite integral region choice problem of double counting rule. 
\end{theorem}

\begin{proof}
The linear independence of $\bar{\mathbf{u}}_1, \cdots , \bar{\mathbf{u}}_l, \bar{\mathbf{u}}_{\infty}$ 
is shown by the similar argument to that for $\mathbf{u}_1, \cdots , \mathbf{u}_l, \mathbf{u}_{\infty}$ 
in the proof of Theorem \ref{thm;KerBasis}. 
Let $\mathbf{y}$ be a kernel solution of  the homomorphism $\Phi _{d2}(D)$. 
By Theorem \ref{thm;rank}, 
the rank 
of the kernel of the homomorphism $\Phi _{d2}(D)$ is $l+1$. 
Then $\mathbf{y}, \bar{\mathbf{u}}_1, \cdots , \bar{\mathbf{u}}_l, \bar{\mathbf{u}}_{\infty}$ 
are linearly dependent, 
since $\bar{\mathbf{u}}_1, \cdots , \bar{\mathbf{u}}_l, \bar{\mathbf{u}}_{\infty}$ are linearly independent. 
By the similar argument to 
that for a kernel solution $\mathbf{x}$ of $\Phi _{a2}(D)$ in the proof of Theorem \ref{thm;KerBasis}, 
$\mathbf{y}$ is a linear combination of $\bar{\mathbf{u}}_1, \cdots , \bar{\mathbf{u}}_l, \bar{\mathbf{u}}_{\infty}$ 
over $\mathbb{Z}$. 
Therefore the set of the kernel solutions $\bar{\mathbf{u}}_1, \cdots , \bar{\mathbf{u}}_l$,  and $\bar{\mathbf{u}}_{\infty}$ 
is a basis of the kernel of the homomorphism $\Phi _{d2}(D)$. 
\end{proof}

Figure \ref{Fig;kerbasisDZRC} gives an example of the basis obtained by Theorem \ref{thm;KerBasisD}. 

\begin{figure}[htbp]
\begin{center}
\includegraphics[height=2.8cm,clip]{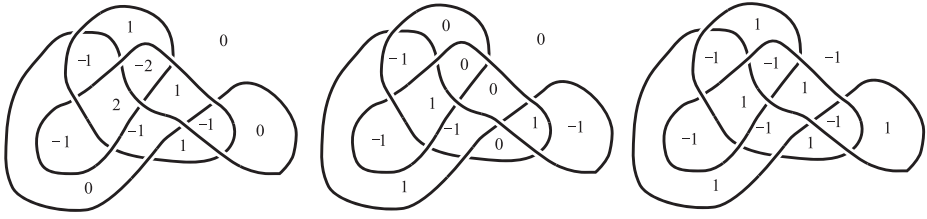}
\end{center}
\caption{The basis $\bar{\mathbf{u}}_1$, $\bar{\mathbf{u}}_2$, $\bar{\mathbf{u}}_{\infty}$.}
\label{Fig;kerbasisDZRC}
\end{figure}

Theorem \ref{thm;KerBasis} and  \ref{thm;KerBasisD} 
are extensions of the result 
due to Hashizume \cite{hashizume2013} 
on a region crossing change. 
Her basis of the kernel of $\mathbb{Z}_2$-homomorphism induced by region crossing changes 
is same as 
the basis given in Theorem \ref{thm;KerBasis} and the basis given in Theorem \ref{thm;KerBasisD} modulo $2$.

\section{Images of homomorphisms induced by integral region choice problems}\label{sect;image}


Let $D$ be a link diagram with at least one crossing. 
For each $i=1, 2$, 
the system of linear equations $A_{di}(D)\mathbf{u}+\mathbf{c}=\mathbf{0}$ 
(resp. $A_{ai}(D)\mathbf{u}+\mathbf{c}=\mathbf{0}$) is solvable 
if and only if 
$\mathbf{c}$ lies in the image of the homomorphism $\Phi _{di}(D)$ (resp. $\Phi _{ai}(D)$). 
In this section, 
we discuss about 
the images of the homomorphisms induced by integral region choice problems. 

By Theorem \ref{thm;aharasuzuki;mat} and \ref{thm;harada;mat}, 
the homomorphisms $\Phi _{d1}(D), \Phi _{d2}(D), \Phi _{a1}(D), \Phi _{a2}(D)$ 
defined in Section \ref{sect;kernel}
are surjective if $D$ is a knot diagram. 
Otherwise they are not surjective in general by Theorem \ref{thm;rank} and \ref{thm;rankS}. 

In \cite{chenggao}, 
Cheng and Gao proved that 
a region crosssing change on a 2-component link diagram is an unknotting operation 
if and only if the linking number is even, 
showing that 
changing two crossings of different components on a 2-component link diagram 
is represented by certain region crossing changes. 

For example, the canonical diagram of $(2,4)$-torus link 
changes to a diagram of the trivial link by one region choice. 
Otherwise there exists an equipment of integers to the crossings on this diagram 
such that the definite and alternating integral region choice problem does not have any solution. 
On the canonical diagram of $(2,4)$-torus link with certain orders of crossincs and regions, 
the definite region choice matrix is 
\[
A_d=
\begin{pmatrix}
1 & 1 & 1 & 1 & 0 & 0 \\
0 & 1 & 1 & 1 & 1 & 0 \\
0 & 1 & 1 & 0 & 1 & 1 \\
1 & 1 & 1 & 0 & 0 & 1 
\end{pmatrix}
\]
and the alternating region choice matrix is 
\[
A_a=
\begin{pmatrix}
-1 & 1 & 1 & -1 & 0  & 0 \\
0  & 1 & 1 & -1 & -1 & 0 \\
0  & 1 & 1 & 0  & -1 & -1 \\
-1 & 1 & 1 & 0  & 0  & -1 
\end{pmatrix}
.\]
The system of linear equations $A_d\mathbf{u}+\mathbf{c}=\mathbf{0}$ has a solution $\mathbf{u} \in \mathbb{Z}^6$ 
if and only if $c_1-c_2+c_3-c_4=0$ holds where $c_i$ is the $i$-th component of $\mathbf{c}\in \mathbb{Z}^4$.  
The system of linear equations $A_a\mathbf{u}+\mathbf{c}=\mathbf{0}$ is solvable 
if and only if $c_1-c_2+c_3-c_4=0$ holds.  
Then in the case $(c_1, c_2, c_3, c_4)=(1,0,0,-1)$, 
any $\mathbf{u} \in \mathbb{Z}^6$ does not hold 
$A_d\mathbf{u}+\mathbf{c}=\mathbf{0}$ or $A_a\mathbf{u}+\mathbf{c}=\mathbf{0}$, 
though we have 
\[
A_d
\begin{pmatrix}
1 \\ 0 \\ 0 \\ 0 \\ 0 \\ 0
\end{pmatrix}
=
\begin{pmatrix}
1 \\ 0 \\ 0 \\ 1
\end{pmatrix}
= - 
\begin{pmatrix}
1 \\ 0 \\ 0 \\ -1
\end{pmatrix}
\in {\mathbb{Z}_2}^4 , \ \ \ 
A_a
\begin{pmatrix}
1 \\ 0 \\ 0 \\ 0 \\ 0 \\ 0
\end{pmatrix}
=
\begin{pmatrix}
-1 \\ 0 \\ 0 \\ -1
\end{pmatrix}
= - 
\begin{pmatrix}
1 \\ 0 \\ 0 \\ -1
\end{pmatrix}
\in {\mathbb{Z}_2}^4 .
\]

In \cite{hashizume2015}, 
Hashizume gave a generating system of the image of the $\mathbb{Z}_2$-homomorphism 
induced by region crossing changes on a link diagram. 
Their results include the following results. 

\begin{lemma}[\cite{chenggao, hashizume2015}]\label{lem;Im2comp}
Let $D$ be a connected diagram of two-component link with $n$ crossings. 
We take two distinct crossings $x,y$ of $D$ of arcs in different link components. 
There exist $\mathbf{v}_{xy} \in \mathbb{Z}_2^{n+2}$ such that 
any components of $A_{d1}(D)\mathbf{v}_{xy} \in {\mathbb{Z}_2}^{n}$ are $0$ 
but the components of $A_{d1}(D)\mathbf{v}_{xy} \in {\mathbb{Z}_2}^{n}$ to $x$ and $y$ are $1$. 
\hfill $\square$
\end{lemma}

\begin{theorem}[\cite{hashizume2015}]\label{thm;ImBasis2comp}
Let $D$ be a connected diagram of two-component link with $n$ crossings. 
The image of the $\mathbb{Z}_2$-homomorphism induced by region crossing changes 
is generated by the elements in ${\mathbb{Z}_2}^n$ of the following two types: 
\begin{enumerate}
\item 
any components are $0$ 
but the only one component corresponding to a crossing in same link component is $1$; 
\item 
any components are $0$ 
but the only two components corresponding to two distinct crossings of distinct link components are $1$. 
\end{enumerate}
\hfill $\square$
\end{theorem}

We note that 
a connected diagram of two-component link 
has at least two crossings because of the Jordan curve theorem.

We extend Lemma \ref{lem;Im2comp} to the alternating integral region choice problem 
as follows, where we define 
$\varepsilon _x=1$ for a positive crossing $x$ 
and $\varepsilon _x=-1$ for a negative crossing $x$. 

\begin{lemma}
\label{lem;Im2compAZRC}
Let $D=D_1\cup D_2$ be a connected diagram 
of two-component oriented link with $n$ crossings, 
where $D_1$ and $D_2$ are sub-diagram of $D$ 
representing the first and the second components respectively. 
We take two distinct crossings $x,y$ of 
$D_1$ and $D_2$. 
We suppose that $D_2$ crosses $D_1$ from the right to the left at $x$. 
\begin{enumerate}
\item 
If $D_2$ crosses $D_1$ from the left to the right at $y$, then 
there exist $\mathbf{v}_{xy} \in \mathbb{Z}^{n+2}$ such that 
any components of $A_{a2}(D)\mathbf{v}_{xy} \in \mathbb{Z}^{n}$ are $0$ 
but the components of $A_{a2}(D)\mathbf{v}_{xy} \in \mathbb{Z}^{n}$ to $x$ and $y$ are  
$\varepsilon _x$ and $\varepsilon _y$. 
\item 
If $D_2$ crosses $D_1$ from the right to the left at $y$, then 
there exist $\mathbf{v}_{xy} \in \mathbb{Z}^{n+2}$ such that 
any components of $A_{a2}(D)\mathbf{v}_{xy} \in \mathbb{Z}^{n}$ are $0$ 
but the components of $A_{a2}(D)\mathbf{v}_{xy} \in \mathbb{Z}^{n}$ to $x$ and $y$ are  
$\varepsilon _x$ and $-\varepsilon _y$.  
\end{enumerate}
\end{lemma}

\begin{figure}[htbp]
\begin{center}
\includegraphics[height=6.5cm,clip]{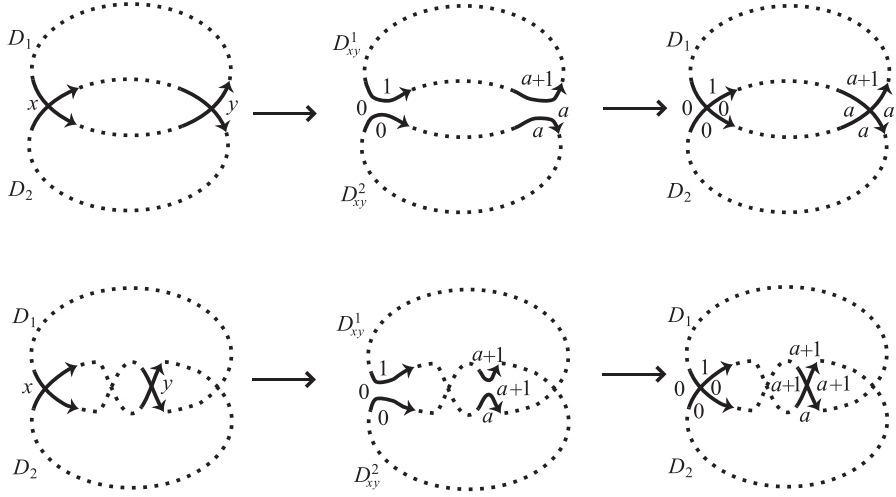}
\end{center}
\caption{Obtaining $\mathbf{v}_{xy}$ in the two cases.}
\label{Fig;2comp}
\end{figure}

\begin{proof}
We splice at $x$. 
Let $\gamma _1$ and $\gamma _2$ be oriented arcs appearing after the splice at $x$ on the obtained diagram. 
We suppose that $\gamma _1$ lies on the left of $\gamma _2$. 
We splice at $y$. 
We obtain 
a new diagram of a two-component link 
as illustrated on the middle of Figure \ref{Fig;2comp}, 
where the cases (1) and (2) are  described in the upper and lower rows respectively. 
We denote the sub-diagram of the link component including the arc $\gamma _i$ by $D_{xy}^i$, $i=1,2$. 
For the diagram $D_{xy}^1 \cup D_{xy}^2$, 
we take the componentwise Alexander numbering associated with $D_{xy}^1$
such that the right and left regions of $\gamma _1$ are assigned $0$ and $1$ respectively. 
We denote by $a$ the integer assigned to the right region of two oriented arcs which appear after the splice at $y$. 
In the case (1), $D_{xy}^1$ includes the left of these two arcs, 
then the region between the arcs is assigned $a$, 
and the left region is assigned $a+1$.  
In the case (2), $D_{xy}^1$ includes the right of these two arcs, 
then the both regions adjacent to the left arc are assigned $a+1$.  
By Lemma \ref{lem;compAlex}, 
this numbering gives 
a kernel solution for $A_{a2}(D_{xy}^1 \cup D_{xy}^2)$ 
if $D_{xy}^1 \cup D_{xy}^2$ has at least one crossing. 
We unsplice $D_{xy}^1 \cup D_{xy}^2$ at $x$ and $y$.  
We assign the same integers to all regions of $D$ as $D_{xy}^1 \cup D_{xy}^2$, 
where the two regions splitting at $x$ are assigned $0$,  
and where the two regions splitting at $y$ are assigned $a$ and $a+1$ in the case (1) and (2) respectively, 
as illustrated on the right of Figure \ref{Fig;2comp}.  
Then the obtained assignment of integers to regions is the desired $\mathbf{v}_{xy} \in \mathbb{Z}^{n+2}$ 
in the both cases. 
\end{proof}

By Lemma \ref{lem;Im2to1} (2) and \ref{lem;Im2compAZRC}, 
we obtain the following result. 

\begin{corollary}
\label{cor;Im2compAZRC}
Let $D=D_1\cup D_2$ be a connected diagram 
of two-component oriented link with $n$ crossings, 
where $D_1$ and $D_2$ are sub-diagram of $D$ 
representing the first and the second components respectively.  
We take two distinct crossings $x,y$ of 
$D_1$ and $D_2$. 
We suppose that $D_2$ crosses $D_1$ from the right to the left at $x$. 
\begin{enumerate}
\item 
If $D_2$ crosses $D_1$ from the left to the right at $y$, then 
there exist $\mathbf{v}_{xy} \in \mathbb{Z}^{n+2}$ such that 
any components of $A_{a1}(D)\mathbf{v}_{xy} \in \mathbb{Z}^{n}$ are $0$ 
but the components of $A_{a1}(D)\mathbf{v}_{xy} \in \mathbb{Z}^{n}$ to $x$ and $y$ are  
$\varepsilon _x$ and $\varepsilon _y$. 
\item 
If $D_2$ crosses $D_1$ from the right to the left at $y$, then 
there exist $\mathbf{v}_{xy} \in \mathbb{Z}^{n+2}$ such that 
any components of $A_{a1}(D)\mathbf{v}_{xy} \in \mathbb{Z}^{n}$ are $0$ 
but the components of $A_{a1}(D)\mathbf{v}_{xy} \in \mathbb{Z}^{n}$ to $x$ and $y$ are  
$\varepsilon _x$ and $-\varepsilon _y$.  
\end{enumerate}
\hfill $\square$
\end{corollary}
 
We note that 
Lemma \ref{lem;Im2comp} is a modulo $2$ reduction of Corollary \ref{cor;Im2compAZRC}. 

To construct a generating system of the image of $\Phi _{a2}(D)$ 
for a connected diagram of a 2-component link $D=D_1\cup D_2$ 
extending Theorem \ref{thm;ImBasis2comp}, 
we do not need all pair of the distinct crossings of $D_1$ and $D_2$,  
since we obtain the following result by easy calculation. 

\begin{corollary}\label{cor;Im2compAZRC2}
Let $D=D_1\cup D_2$ be a connected diagram 
of two-component oriented link with $n$ crossings, 
where $D_1$ and $D_2$ are sub-diagram of $D$ 
representing the first and the second components respectively.  
We suppose that 
there exist three distinct crossings $x,y,z$ of 
$D_1$ and $D_2$,  
and that $D_2$ crosses $D_1$ from the right to the left at $x$. 
Let $\mathbf{v}_{xy}, \mathbf{v}_{xz} \in \mathbb{Z}^n$ be obtained by Lemma \ref{lem;Im2compAZRC}.  
\begin{enumerate}
\item 
If $D_2$ crosses $D_1$ from the left to the right at $y$ and $z$, 
then any components of  $A_{a2}(D) (\mathbf{v}_{xy}-\mathbf{v}_{xz})$ are $0$ 
but the components of  $A_{a2}(D) (\mathbf{v}_{xy}-\mathbf{v}_{xz})$ to $y$ and $z$ are 
$\varepsilon _y$ and $-\varepsilon _z$ respectively. 
\item 
If $D_2$ crosses $D_1$ from the left to the right at $y$
 and  from the right to the left at $z$, 
then any components of  $A_{a2}(D) (\mathbf{v}_{xy}-\mathbf{v}_{xz})$ are $0$ 
but the components of  $A_{a2}(D) (\mathbf{v}_{xy}-\mathbf{v}_{xz})$ to $y$ and $z$ are 
$\varepsilon _y$ and $\varepsilon _z$ respectively. 
\item 
If $D_2$ crosses $D_1$ from the right to the left at $y$ and $z$, 
then any components of  $A_{a2}(D) (\mathbf{v}_{xz}-\mathbf{v}_{xy})$ are $0$ 
but the components of  $A_{a2}(D) (\mathbf{v}_{xz}-\mathbf{v}_{xy})$ to $y$ and $z$ are 
$\varepsilon _y$ and $-\varepsilon _z$ respectively. 
\end{enumerate}
\hfill $\square$
\end{corollary}

We note that 
the above $\mathbf{v}_{xy}-\mathbf{v}_{xz}$ or $\mathbf{v}_{xz}-\mathbf{v}_{xy}$ 
are not equal to 
$\mathbf{v}_{yz}$ obtained in the proof of Lemma \ref{lem;Im2compAZRC} generally. 

We obtain a basis of 
the image of the homomorphism of the alternating integral region choice problem as belows. 

\begin{theorem}\label{thm;ImBasis2compAZRC}
Let $D=D_1\cup D_2$ be a connected diagram 
of two-component oriented link with $n$ crossings $x_1, \cdots , x_n$, 
where $D_1$ and $D_2$ are sub-diagram of $D$ 
representing the first and the second components respectively.  
We suppose that 
each of $x_1, \cdots , x_k$ ($k<n$) is a crossing in $D_1$ or a crossing in $D_2$,  
$x_{k+1}, \cdots , x_n$ are crossings of $D_1$ and $D_2$, 
and $D_2$ crosses $D_1$ from the right to the left at $x_n$. 
We take $\mathbf{e}_1, \cdots , \mathbf{e}_{n-1} \in \mathbb{Z}^n$ as belows: 
\begin{enumerate}
\item
for $i=1, \cdots, k$, 
let $\mathbf{e}_i$ be the element of $\mathbb{Z}^n$ 
such that 
the $i$-th component is $1$ and the others are $0$;
\item 
for $i=k+1, \cdots , n-1$, 
let $\mathbf{e}_i$ be the element of $\mathbb{Z}^n$ 
such that 
the $n$-th component is $\varepsilon _{x_n}$ and 
the $i$-th component is  $\varepsilon _{x_i}$ (resp. $-\varepsilon _{x_i}$) 
if $D_2$ crosses $D_1$ at $x_i$ from the left to the right (resp. from the right to the left), 
and that the others are $0$. 
\end{enumerate}
Then the set of $\mathbf{e}_1, \cdots , \mathbf{e}_{n-1}$ is 
a basis of the image of the homomorphism 
induced by the alternating integral region choice problem. 
Therefore the systems of linear equations 
$A_{a1}(D_1\cup D_2) \mathbf{u}+\mathbf{c}=\mathbf{0}$ and 
$A_{a2}(D_1\cup D_2) \mathbf{w}+\mathbf{c}=\mathbf{0}$ have 
solutions $\mathbf{u}, \mathbf{w} \in \mathbb{Z}^{n+2}$ 
if and only if 
$\mathbf{c} \in \mathbb{Z}^n$ is 
a linear combination of 
$\mathbf{e}_1, \cdots , \mathbf{e}_{n-1}$.  
\end{theorem}

\begin{proof}
By Theorem \ref{thm;add1AZRC} and Lemma \ref{lem;Im2compAZRC}, 
$\mathbf{e}_1, \cdots , \mathbf{e}_{n-1}$ are elements in 
the image of the homomorphism $\Phi _{a2} (D_1\cup D_2)$. 
They are linearly independent by the construction. 
Let $\mathbf{c}$ be an element of the image of the homomorphism $\Phi _{a2} (D_1\cup D_2)$. 
By Theorem \ref{thm;rank}, 
the rank of the image of the homomorphism $\Phi _{a2} (D_1\cup D_2)$ is $n+1-2=n-1$. 
Hence $\mathbf{c}, \mathbf{e}_1, \cdots , \mathbf{e}_{n-1}$ are linearly dependent 
since $\mathbf{e}_1, \cdots , \mathbf{e}_{n-1}$ are linearly independent. 
By the construction of $\mathbf{e}_1, \cdots , \mathbf{e}_{n-1}$, 
it is shown that $\mathbf{c}$ is a linear combination of  $\mathbf{e}_1, \cdots , \mathbf{e}_{n-1}$ over $\mathbb{Z}$. 
Then the set of $\mathbf{e}_1, \cdots , \mathbf{e}_{n-1}$ is 
a basis of the image of the homomorphism  $\Phi _{a2}(D_1\cup D_2)$. 

By Theorem \ref{thm;image} (2), 
the image of the homomorphism $\Phi _{a1}(D_1\cup D_2)$
coincides with the image of the homomorphism $\Phi _{a2}(D_1\cup D_2)$. 
\end{proof}

We note that the modulo $2$ reduction of Theorem \ref{thm;ImBasis2compAZRC} implies 
Theorem \ref{thm;ImBasis2comp}. 

From the above basis, 
we obtain a basis of the image of the homomorphism of 
the definite integral region choice problem as belows. 
Let $D=D_1\cup D_2$ be a connected diagram 
of two-component oriented link with $n$ crossings $x_1, \cdots , x_n$. 
Let $R_1, \cdots , R_{n+2}$ be the regions of $D$. 
We suppose that 
each of $x_1, \cdots , x_k$ ($k<n$) is a crossing in $D_1$ or a crossing in $D_2$, 
$x_{k+1}, \cdots , x_n$ are crossings of $D_1$ and $D_2$, 
and $D_2$ crosses $D_1$ from the right to the left at $x_n$. 
For each of $\mathbf{e}_1, \cdots , \mathbf{e}_{n-1}$ obtained by Theorem \ref{thm;ImBasis2compAZRC}, 
there exists a solution $\mathbf{v}_i \in \mathbb{Z}^{n+2}$ of $A_{a2}(D)\mathbf{v}_i=\mathbf{e}_i$. 
We take the checkerboard coloring such that 
the left and right regions of both oriented arcs crossing at $x_n$ are assigned $0$ as 
  \begin{minipage}{15\unitlength}
    \begin{picture}(15,15)
      \put(0,0){\vector(1,1){15}}
      \qbezier(15,0)(15,0)(10,5)
      \qbezier(5,10)(0,15)(0,15)
      \put(0,15){\vector(-1,1){0}}
    \put(10,5){{\footnotesize 0}}
    \put(0,5){{\footnotesize 0}}
    \put(5,10){{\footnotesize 1}}
    \put(5,0){{\footnotesize 1}}
    \end{picture}
  \end{minipage}
or 
  \begin{minipage}{15\unitlength}
    \begin{picture}(15,15)
      \put(15,0){\vector(-1,1){15}}
      \qbezier(0,0)(0,0)(5,5)
      \qbezier(10,10)(15,15)(15,15)
      \put(15,15){\vector(1,1){0}}
    \put(10,5){{\footnotesize 0}}
    \put(0,5){{\footnotesize 0}}
    \put(5,10){{\footnotesize 1}}
    \put(5,0){{\footnotesize 1}}
    \end{picture}
  \end{minipage}
. 
For each $i= 1, \cdots , k$, 
there exists a solution $\bar{\mathbf{v}}_i \in \mathbb{Z}^{n+2}$ of $A_{d2}(D)\bar{\mathbf{v}}_i=\mathbf{e}_i$ 
by Theorem \ref{thm;add1DZRC}. 
For each $i=k+1, \cdots , n-1$, 
we multiply the $j$-th component of $\mathbf{v}_i$ by $-1$ 
if the region $R_j$ is assigned the checkerboard index $1$. 
We denote by $\bar{\mathbf{v}}_i$ the obtained element of $\mathbb{Z}^{n+2}$ from $\mathbf{v}_i$. 
Let $\bar{\mathbf{e}}_i=A_{d2}(D)\bar{\mathbf{v}}_i$. 
Then the $n$-th component of $\bar{\mathbf{e}}_i$ becomes $1$, 
the $i$-th component becomes $1$ or $-1$, 
and the others are $0$. 

\begin{theorem}
The set of the above $\mathbf{e}_1, \cdots , \mathbf{e}_k, \bar{\mathbf{e}}_{k+1}, \cdots , \bar{\mathbf{e}}_{n-1}$ 
is a a basis of the image of the homomorphism 
induced by the definite integral region choice problem 
on the connected diagram of two-component link $D=D_1\cup D_2$. 
Therefore the systems of linear equations 
$A_{d1}(D_1\cup D_2) \mathbf{u}+\mathbf{c}=\mathbf{0}$ and 
$A_{d2}(D_1\cup D_2) \mathbf{w}+\mathbf{c}=\mathbf{0}$ have 
solutions $\mathbf{u}, \mathbf{w} \in \mathbb{Z}^{n+2}$ 
if and only if 
$\mathbf{c} \in \mathbb{Z}^n$ is 
a linear combination of 
$\mathbf{e}_1, \cdots , \mathbf{e}_k, \bar{\mathbf{e}}_{k+1}, \cdots , \bar{\mathbf{e}}_{n-1}$.  
\end{theorem}

\begin{proof}
By the construction, 
$\mathbf{e}_1, \cdots , \mathbf{e}_k, \bar{\mathbf{e}}_{k+1}, \cdots , \bar{\mathbf{e}}_{n-1}$ 
are elements in the image of the homomorphism $\Phi _{d2}(D_1\cup D_2)$ 
and linearly independent. 
Let $\mathbf{c}$ be an element of the image of the homomorphism $\Phi _{d2}(D_1\cup D_2)$. 
By Theorem \ref{thm;rank}, 
the rank of the image of the homomorphism $\Phi _{d2}(D_1\cup D_2)$ is $n+1-2=n-1$. 
Hence $\mathbf{c}, \mathbf{e}_1, \cdots , \mathbf{e}_k, \bar{\mathbf{e}}_{k+1}, \cdots , \bar{\mathbf{e}}_{n-1}$ 
are linearly dependent 
since $\mathbf{e}_1, \cdots , \mathbf{e}_k, \bar{\mathbf{e}}_{k+1}, \cdots , \bar{\mathbf{e}}_{n-1}$ are linearly independent. 
By the construction of $\mathbf{e}_1, \cdots , \mathbf{e}_k, \bar{\mathbf{e}}_{k+1}, \cdots , \bar{\mathbf{e}}_{n-1}$, 
it is shown that $\mathbf{c}$ is  
a linear combination of $\mathbf{e}_1, \cdots , \mathbf{e}_k, \bar{\mathbf{e}}_{k+1}, \cdots , \bar{\mathbf{e}}_{n-1}$ 
over $\mathbb{Z}$.
Then the set of $\mathbf{e}_1, \cdots , \mathbf{e}_k, \bar{\mathbf{e}}_{k+1}, \cdots , \bar{\mathbf{e}}_{n-1}$ 
is a basis of the image of the homomorphism  $\Phi _{d2}(D_1\cup D_2)$. 

By Theorem \ref{thm;image} (1), 
the image of the homomorphism $\Phi _{d1}(D_1\cup D_2)$
coincides with the image of the homomorphism $\Phi _{d2}(D_1\cup D_2)$. 
\end{proof}

\begin{remark}
In \cite{hashizume2015}, 
Hashizume gave a generating system of the image of the $\mathbb{Z}_2$-homomorphism 
induced by region crossing changes
for a link diagram 
with arbitrary number of link components. 
Her generating system includes 
that in Theorem \ref{thm;ImBasis2comp}. 
For each of the $\mathbb{Z}$-homomorphisms 
induced by integral region choice problems on a link diagram 
with arbitrary number of link components, 
we are finding a basis of the image.  
\end{remark}

\end{document}